\documentclass{amsart}
\usepackage{amsmath,amssymb}
\usepackage{graphicx,subfigure}

\usepackage{color}

\newtheorem{theorem}{Theorem}[section]
\newtheorem{proposition}[theorem]{Proposition}
\newtheorem{corollary}[theorem]{Corollary}

\newtheorem{lemma}[theorem]{Lemma}

\theoremstyle{definition}
 \newtheorem{definition}[theorem]{Definition}

\newtheorem*{definition*}{Definition}

\newcommand\Lab{\La^\bullet}
\newcommand\LanI{\La^1_{x,\mathrm{nor}}}
\newcommand\LanII{\La^2_{x,\mathrm{nor}}}
\newcommand\LatI{\La^1_{x,\mathrm{tan}}}
\newcommand\LatII{\La^2_{x,\mathrm{tan}}}
\newcommand\Latp{\La^p_{x,\mathrm{tan}}}

\newcommand\ext{_{\mathrm{ext}}}

\theoremstyle{remark}
\newtheorem{remark}[theorem]{Remark}

\numberwithin{equation}{section}

\newcommand{\al}{\alpha}
\newcommand{\be}{\beta}
\newcommand{\de}{\delta}
\newcommand{\ep}{\varepsilon}

\newcommand{\ga}{\gamma}

\newcommand{\la}{\lambda}
\newcommand{\om}{\omega}
\newcommand{\si}{\sigma}

\newcommand{\vp}{\varphi}

\newcommand{\De}{\Delta}

\newcommand{\La}{\Lambda}
\newcommand{\Si}{\Sigma}
\newcommand{\Om}{\Omega}

\def\RR{\mathbb{R}}

\def\ZZ{\mathbb{Z}}

\def\TT{\mathbb{T}}

\newcommand{\calT}{\mathcal T}
\newcommand{\cD}{{\mathcal D}}

\newcommand{\cH}{{\mathcal H}}

\newcommand{\cL}{{\mathcal L}}

\newcommand{\cO}{{\mathcal O}}

\newcommand{\cR}{{\mathcal R}}

\newcommand{\cT}{{\mathcal T}}

\newcommand{\pd}{\partial}
\newcommand\minus\backslash
\newcommand{\id}{{\rm id}}

\newcommand\lan\langle
\newcommand\ran\rangle

\DeclareMathOperator\Div{div}

\renewcommand\leq\leqslant
\renewcommand\geq\geqslant
\newlength{\intwidth}

\newcommand\BOm{\overline\Om}

 \DeclareMathOperator\curl{curl}

\begin{document}

\title[MHD equilibria with nonconstant pressure]{MHD
  equilibria with nonconstant pressure\\ in nondegenerate toroidal
  domains}

\author{Alberto Enciso}
\address{Instituto de Ciencias Matem\'aticas, Consejo Superior de
  Investigaciones Cient\'\i ficas, 28049 Madrid, Spain}
\email{aenciso@icmat.es}

\author{Alejandro Luque}
\address{Department of Mathematics, Uppsala University, 751 06
  Uppsala, Sweden}
\email{alejandro.luque.math@gmail.com}

\author{Daniel Peralta-Salas}
\address{Instituto de Ciencias Matem\'aticas, Consejo Superior de
  Investigaciones Cient\'\i ficas, 28049 Madrid, Spain}
\email{dperalta@icmat.es}

%%    General info
%\subjclass[2010]{35B38, 58J05, 58K45}
%\date{\today}
%
%\keywords{ }
%
\begin{abstract}
  We prove the existence of piecewise smooth MHD equilibria in three-dimensional
  toroidal domains of~$\RR^3$ where the pressure is constant on the boundary but not in the interior. The pressure is piecewise
  constant and the plasma current exhibits an arbitrary number of current
  sheets. We also establish the existence of free boundary steady states surrounded by vacuum with an external surface current. The toroidal domains where these equilibria are shown to
  exist do not need to be small perturbations of an axisymmetric
  domain, and in fact they
  can have any knotted topology. The building blocks we use in our
  construction are analytic toroidal domains satisfying a certain
  nondegeneracy condition, which roughly states that there exists a
  force-free field that is ergodic on the surface of the domain. The
  proof involves three main ingredients: a gluing construction of
  piecewise smooth MHD
  equilibria, a Hamilton--Jacobi equation on the two-dimensional torus
  that can be understood as a nonlinear deformation of the
  cohomological equation (so the nondegeneracy assumption plays a
  major role in the corresponding analysis), and a new KAM theorem tailored for the study
  of divergence-free fields in three dimensions whose Poincar\'e map
  cannot be computed explicitly.
\end{abstract}
\maketitle

\section{Introduction}

The magnetohydrodynamics (MHD) equilibria in a smooth bounded domain
$\Om\subset\RR^3$ are often described by the solutions to the system
of equations
\begin{subequations}\label{fixed}
\begin{align}\label{Eq.MHD}
B\times \curl B+\nabla P&=0 \qquad\text{ in }\Om\,,\\
\label{eq.MHD2}\Div B&=0 \qquad\text{ in }\Om\,,\\
\label{eq.MHD3}B\cdot N&=0 \qquad\text{ on } \partial\Om\,.
\end{align}
Here the vector field $B(x)$ is the steady magnetic field of a
perfectly conducting plasma, the scalar function $P(x)$ is the plasma
pressure and $N(x)$ denotes the outer unit normal at a boundary
point $x\in\pd\Om$. It is also customary to assume that $P$ is
constant on the boundary $\pd\Om$, so we can set
\begin{equation}\label{Eq.MHD4}
P=0 \qquad \text{ on } \partial\Om\,.
\end{equation}
\end{subequations}

Equations~\eqref{fixed}, which we will referred to as
the {\em fixed boundary problem}\/, model the equilibrium
configurations of a plasma
confined in the fixed magnetic domain~$\Om$.
Because of this connection, the
most interesting case is when $\Om$ is a toroidal domain, that is,
when the boundary of the domain is diffeomorphic to a
torus. It is worth noting that
these equations also describe stationary
solutions to the 3D Euler equations in fluid mechanics, with $B$
playing the role of the velocity field of the fluid and $P$ being
the negative of the Bernoulli function (see e.g.~\cite{AK}).

In this paper we will be concerned with the
problem of finding toroidal domains that admit MHD equilibria whose
pressure is constant on the boundary but not in the interior. This
problem can be traced back at least to the work of
H.~Grad in the 1960s, who conjectured~\cite{Grad} that no smooth solutions fibring the domain with toroidal surfaces exist
unless the domain is axially symmetric. An important somewhat related result, due to Bruno and Laurence~\cite{T09} in the 1990s, is the
existence of weak solutions with nonconstant stepped-pressure in nonsymmetric toroidal domains that are small perturbations of an axisymmetric solid torus.
A very illuminating numerical implementation of this model suggesting the existence of stepped-pressure equilibria in toroidal domains far from axisymmetric was developed in~\cite{HHD07,HDDHMNL12}. See also~\cite{CDG} for the construction of smooth solutions with nonconstant pressure, but subject to a small force, on toroidal domains that are close to be axisymmetric. However, Grad's influential conjecture remains wide open. A comprehensive recent account of the subject can be found in~\cite{Meiss,Helander}.

It is important to remark that in the context of plasma physics, the pressure profile $P$ (usually described as a function of the magnetic flux) and the rotational transform profile (which is, essentially, the ratio of the components of the frequency vector of the magnetic field on each invariant toroidal surface) are prescribed \emph{a priori}, as inputs of the computations. The relevance of these prescriptions lays in the fact that they yield uniqueness in the solutions of Equations~\eqref{fixed}, if they exist. In contrast, in this work we will be interested in the existence of stepped-pressure MHD equilibria in domains very far from symmetric, without imposing any profile, and then the pressure and the rotational transform profiles show up \emph{a posteriori}. The same comment applies to the free boundary problem that is next stated.

Another related equation that appears in plasma physics,
particularly concerning the design of plasma confinement devices for
nuclear power generation, describes a free boundary steady state
surrounded by vacuum with an external current $J\ext$. In terms of the
interior and exterior magnetic fields, $B$ and $B\ext$, this system
reads in this case as
\begin{subequations}\label{free}
\begin{align}
B\times \curl B+\nabla P&=0 \qquad\text{ in }\Om\,, \label{Eq.FBP1}\\
\Div B&=0 \qquad\text{ in }\Om\,, \label{Eq.FBP2}\\
  \Div B\ext&=0 \qquad\text{ in }\RR^3\backslash\overline{\Om}\,, \label{Eq.FBP3}\\
\curl  B\ext&=J\ext \qquad\text{ in
              }\RR^3\backslash\overline{\Om}\,, \label{Eq.FBP4}\\
    (B-B\ext)\cdot N&=0 \qquad\text{ on } \partial\Om\,, \label{Eq.FBP5}\\
  |B|^2-|B\ext|^2&=0 \qquad\text{ on } \partial\Om\,, \label{Eq.FBP6}\\
  B\ext&\to0 \qquad\,\text{as } |x|\to\infty\,. \label{Eq.FBP7}
\end{align}
For simplicity, we have assumed that $P=0$ on $\pd\Om$. The jump condition~\eqref{Eq.FBP6} uses that the (hydrodynamic) pressure in the
vacuum is simply $\frac12|B\ext|^2$. In this work we shall assume that
the external current is a current sheet, i.e.,
\begin{equation}\label{Jext}
J\ext = J\, dS
\end{equation}
where $dS$ is the surface measure on the boundary~$\pd\Om'$ of a certain domain $\Om'$
enclosing $\overline{\Om}$ and $J:\pd\Om'\to\RR^3$ is a tangent vector field on the surface. We want to stress that in applications in the context of plasma confinement (stellarators), the external current $J\ext$ is generated by a certain number of current coils, but the techniques we use in this work do not allow us to deal with this case (see Remark~\ref{R.coils}). We will additionally impose the
tangency condition
\begin{equation}\label{Eq.FBP8}
B\cdot N =0 \qquad\text{ on } \partial\Om
\end{equation}
\end{subequations}
and refer to the system of equations~\eqref{free} as the {\em free boundary problem}\/.

We observe that a related, although different, free boundary problem for MHD equilibria has been considered in the literature, see e.g.~\cite{Temam,Beres} and references therein. In these works, the problem consists in determining the shape of an axisymmetric plasma region, which is surrounded by vacuum and contained in a given axisymmetric shell. It is assumed that the shell is a perfect conductor, so no external current is introduced in the vacuum region; moreover the vacuum field is assumed to be tangent to the boundaries of the shell and of the plasma, and the continuity of the hydrodynamic pressure is also imposed (in order to have a weak solution of the MHD equations across the plasma boundary). In contrast, in the free boundary problem that we consider above, the vacuum vessel is not a perfect conductor, so this leads to consider MHD equilibria on a fixed plasma region $\Om$ (the fixed boundary problem), and look for an external current $J\ext$ in the vacuum whose corresponding magnetic field $B\ext$ forces the plasma to be in equilibrium on the whole space (the reason for the boundary conditions~\eqref{Eq.FBP5} and~\eqref{Eq.FBP6}).

The main result of this article is the existence of piecewise smooth
MHD equilibria with nonconstant stepped-pressure in a wide range of
toroidal domains, which can be very different from an axisymmetric
domain. The same philosophy works, with only minor modifications,
for the fixed and free boundary problems. The equilibria we construct are not~$C^1$, like those of Bruno
and Laurence for almost-axisymmetric domains, and in fact they feature singular
current sheets (cf.~Remark~\ref{R:current}). The toroidal domains we consider can even be knotted in an arbitrarily
complicated fashion. Specifically, the result applies to any toroidal
domain with an analytic boundary satisfying a certain nondegeneracy
assumption, which enables us to employ KAM-theoretic techniques in a
certain step of the proof.

\subsection{Nondegenerate toroidal domains}

To define the nondegeneracy condition that our toroidal domains must
satisfy, we need to introduce some notation. We recall that by toroidal domain we mean that the domain is diffeomorphic to a solid torus. Firstly, recall that
Hodge theory ensures the existence of a unique (modulo a
multiplicative constant) solution to the boundary problem
\[
\curl h=0\quad\text{in }\Om\,, \qquad \Div h=0\quad\text{in
}\Om\,, \qquad h\cdot N=0 \quad\text{on }\pd\Om
\]
on a toroidal domain~$\Om\subset\RR^3$. We refer to~$h$ as the {\em
  harmonic field}\/ of~$\Om$.

{\em Beltrami fields}\/ are solutions to
the equation
\begin{equation}\label{eqsB}
\curl B=\lambda B\quad\text{in }\Om\,, \qquad B\cdot N=0 \quad\text{on }\pd\Om
\end{equation}
for some nonzero real constant~$\la$. The space of Beltrami fields on
a toroidal domain~$\Om$ is infinite dimensional. Specifically, as the
curl defines a self-adjoint operator with dense domain in the space $L^2_{\Div,h}(\Om)$ of
square-integrable divergence-free fields that are $L^2$-orthogonal
to the harmonic field~$h$, it is standard that there is an orthogonal
basis of $L^2_{\Div,h}(\Om)$ consisting of
Beltrami fields with zero harmonic part, i.e., solutions $\{B^n\}_{n=1}^\infty$ to
Equation~\eqref{eqsB} for certain nonzero constants
$\{\la^n\}_{n=1}^\infty$
that satisfy
the additional constraint
\[
\int_{\Om}B^n\cdot h\, dx=0\,.
\]
Note that $|\la^n|\to\infty$ as $n\to\infty$.
Moreover, for all
$\la\in \RR\backslash\{\la^n
\}_{n=1}^\infty$, there is a unique
Beltrami field with parameter~$\la$ and prescribed harmonic part, that
is, a unique solution to the boundary problem~\eqref{eqsB} subject to
the additional constraint
\[
\int_{\Om}B\cdot h\, dx=1\,.
\]
All along this paper we shall use the term \emph{eigenvalue} when referring to the constant proportionality factor $\la$ of a Beltrami field, even taking into account that, strictly speaking, $\la$ is an eigenvalue of curl only when it belongs to the aforementioned discrete set of points $\{\la^n\}_{n=1}^\infty$.

In both cases (that is, when~$B$ has a prescribed harmonic part and when $B=B^n$ for some~$n$), the trace to the boundary of the Beltrami field is a smooth vector field tangent to the
(embedded, nonsymmetric, possibly knotted) toroidal
surface~$\pd\Om$. In any case, $\pd\Om$ is an invariant torus of $B$.
Now recall that this invariant torus is {\em Diophantine}\/ with frequency vector~$\om\in\RR^2$ if there exist global coordinates $\vp: \pd\Om\to\TT^2$, with $\TT:=\RR/2\pi\ZZ$, such that the restriction of the field~$B$ to~$\partial\Om$ reads in these coordinates as
\begin{equation}\label{eq.torus}
B|_{\partial\Om}=\om_1\, \pd_{\vp_1}+\om_2\, \pd_{\vp_2}\,,
\end{equation}
and $\om$ is a
Diophantine vector. This means that there exist constants~$\ga>0$ and~$\tau>1$ such that
\begin{equation}\label{eq:diof}
|k\cdot \omega| \geq {\gamma}|k|^{-\tau}
\end{equation}
for any vector with integral components
$k\in\ZZ^2\backslash\{0\}$. It is well known that the Diophantine property (possibly with different $\gamma$) of the frequency vector $\om$ is independent of the choice of coordinates~$\vp$. We recall that $\{\pd_{\vp_1},\pd_{\vp_2}\}$ stands for the holonomic (or coordinate) basis associated to the coordinate system $\vp: \pd\Om\to\TT^2$; as usual, the notation $\pd_{\vp_j}$ means that the coordinate basis vector field can be identified with the partial derivative operator (under the usual interpretation of vector fields as differential operators acting on functions).
%The ratio~$\om_1/\om_2$ modulo~$1$, which is a Diophantine number, is
%independent of the choice of coordinates~$\vp$.

In this paper, we will say that a toroidal domain~$\Om\subset\RR^3$ is
{\em nondegenerate of type I or II}\/ if there is a Beltrami field
on~$\Om$ for which the boundary~$\pd\Om$ is a Diophantine invariant
torus and if the determinant of certain $2\times 2$ constant matrices
is not zero (type I) or not equal to a certain constant depending on the
Beltrami field (type II). To
streamline the Introduction, the
expression of these matrices (which are the average on~$\TT^2$ of
matrix-valued quantities involving the specific Beltrami
field, the associated Diophantine frequency vector and the linearizing
coordinates~$\vp$) is relegated to Definitions~\ref{D:torus} and~\ref{D:torusII} in the
main text. To get some intuition about the meaning of this condition, recall
that Beltrami fields appear in the
plasma physics literature as {force-free fields}\/ with constant
factor, so the nondegeneracy condition can be heuristically understood as the existence of a generic
force-free field on the domain whose rotational transform is irrational on the boundary. For concreteness, we shall refer to a Beltrami
field with this property as a {\em nondegenerate Beltrami field of type I or II}\/.

Some observations are in order. Firstly, since
there are infinitely many curl eigenfields that do not necessarily
vanish on the toroidal surface~$\pd\Om$, and since the set of
Diophantine vectors has full
measure, it is natural to conjecture that a ``generic'' toroidal
domain should be nondegenerate, in this sense. However, genericity
questions for vector-valued eigenvalue problems are remarkably
hard~\cite{Uhlenbeck,TAMS} and we have not been able to rigorously establish this
claim. A particular case that one can study in detail is the class of
thin toroidal domains, which one can understand as thickenings of an
arbitrary closed curve in space. There one has a
good understanding on the structure of harmonic fields, which can be
used to rigorously show that thin toroidal domains are indeed
nondegenerate. Details are given in Proposition~\ref{P:isot}. A
concrete consequence of this fact is that there
certainly are nondegenerate toroidal domains of any knot type, which
are obviously not small perturbations of an axisymmetric domain. The
boundary can be chosen to be smooth, and even analytic.

\subsection{Statement of the main results}

We are now ready to present our main result about the existence of
MHD equilibria in nondegenerate toroidal domains.

In the context of
the fixed boundary problem, the result can be stated as follows:

\begin{theorem}[Fixed boundary MHD equilibria]\label{T:main}
Let $\Om_1\subset\RR^3$ be a nondegenerate toroidal domain of type I with
analytic boundary, and let $B_1$ be a
nondegenerate Beltrami field of type I with eigenvalue $\la_1$ on~$\Om_1$, in the sense of
Definition~\ref{D:torus}. Then, for any $N>1$ and almost all
$(\la_2,\dots,\la_N)\in\RR^{N-1}$ with $\la_j\neq\la_k$, the following statements hold:
\begin{enumerate}
\item There exists a collection of ``nested'' nondegenerate toroidal domains of type I
  $\{\Om_k\}_{k=2}^N$ with analytic boundary, all of them diffeomorphic
  to~$\Om_1$, with the property that $\overline{\Om_{k-1}}\subset\Om_k$ for all $2\leq k\leq N$ (see Figure~\ref{fig:nested}).

\item There is a piecewise smooth MHD equilibrium $(B,P)$ in the fixed
  domain~$\Om_N$, which satisfies
  Equations~\eqref{fixed} in the weak sense. In particular, $P+\frac12|B|^2$ is continuous across the surfaces $\pd\Om_k$ for $1\leq k\leq N-1$.

\item For each $1\leq k\leq N$, the magnetic field and the pressure
  satisfy
  \[
\curl B = \la_k B \qquad\text{and}\qquad P= p_k
  \]
  in $\Om_k\backslash\overline{\Om_{k-1}}$. Here $\{p_k\}_{k=1}^{N-1}$
  are distinct nonzero constants, $p_N:=0$ and we have set
  $\Om_0:=\emptyset$. Furthermore, $B=B_1$ in~$\Om_1$.
\end{enumerate}
\end{theorem}
\begin{remark}
The full measure set of values $(\la_2,\dots,\la_N)$ for which the theorem holds is described as the whole $\RR^{N-1}$ minus a finite set of hyperplanes; the values of $\lambda_j$ that must be avoided are given by Equation~\eqref{lambda}, and depend on the nondegenerate Beltrami field $B_1$ and the pressure jumps in a very nontrivial way.
\end{remark}

Likewise, for the free boundary problem, our main result can be stated
as follows. In the statement, since $B$ is a Beltrami field, it solves Equation~\eqref{Eq.FBP1} in $\Om$ for any constant $P$, whose value is inessential and can be set to be $0$. Since the plasma pressure of the vacuum is  $0$ as well, this yields free boundary force-free equilibria.

\begin{theorem}[Free boundary force free equilibria]\label{T:main2}
Let $\Om\subset\RR^3$ be a nondegenerate toroidal domain of type II with
analytic boundary, and let $B$ be a
nondegenerate Beltrami field of type II with eigenvalue $\la$ on~$\Om$, in the sense of
Definition~\ref{D:torusII}. Then there exists an external magnetic
field $B\ext$ and a current sheet of the
form $J\ext = J\, dS$, where $dS$ is the surface measure on the boundary of
an analytic domain~$\Om'$ that is diffeomorphic to~$\Om$ and encloses~$\BOm$, and
where $J$ is an analytic tangent vector field on~$\pd\Om'$, such that
$(B,B\ext,J\ext,\Om)$ is a solution of the free boundary
problem~\eqref{free} with $P=0$ in $\Om$.
\end{theorem}
\begin{remark}\label{R.coils}
The external surface current $J\ext$ can be approximated (in the sense of distributions) by a current $\tilde J\ext$ whose support consists of finitely many closed curves on $\pd\Om'$, see e.g.~\cite{Ulam}. Then, $\tilde J\ext$ represents the current produced by a number of external coils (and hence it cannot be longer represented as $J\,dS$ for some smooth tangent vector field on $\pd\Om'$). However, $(B,\tilde B\ext,\tilde J\ext,\Om)$ is not a solution of the free boundary problem~\eqref{free} because the boundary terms $\tilde B\ext\cdot N$ and $|\tilde B\ext|^2-|B|^2$ on $\pd\Om$ are small but not necessarily zero.
\end{remark}
\begin{remark}\label{R.freeb}
A minor modification of the proof of Theorem~\ref{T:main2} using Theorem~\ref{L:cohom}, allows us to construct free boundary MHD equilibria with constant pressure $P=c\neq0$ in $\Om$, provided that $|c|$ is small enough. In this case, the boundary condition~\eqref{Eq.FBP6} is replaced by the identity $|B\ext|^2-|B|^2=2c$ on $\pd\Om$, and $\Om$ needs to be a nondegenerate toroidal domain of both types I and II. Since the pressure of the vacuum region is~$0$ and $c\neq0$, this modified construction yields stepped-pressure free boundary equilibria in $\RR^3$.   
\end{remark}

A consequence of these theorems and of the above discussion
about nondegenerate domains is the existence of piecewise smooth
MHD equilibria with nonconstant pressure and (fixed or free) toroidal boundaries of any knot
type. A precise statement is given in Corollary~\ref{C.main}. For the
benefit of the reader, in Section~\ref{S:weak} we recall the definition of weak solutions to the
system~\eqref{Eq.MHD}--\eqref{eq.MHD3}, which is required to make
sense of MHD equilibria that are only piecewise smooth.

\begin{figure}[t]
\centering{
\fontsize{9pt}{11pt}\selectfont
\def\svgwidth{3.333in}
\resizebox{75mm}{!}{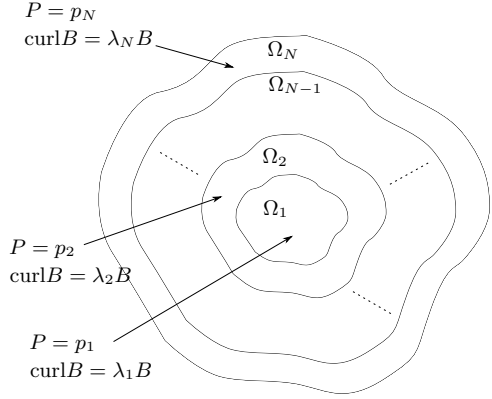}
\caption{A cross section of the nested toroidal domains.}
\label{fig:nested}
}
\end{figure}

It is worth mentioning that a minor modification of the proof of Theorem~\ref{T:main}
permits to prove the existence of Lipschitz continuous force-free
fields with a nonconstant proportionality factor on toroidal domains
of complicated geometry. Details are provided in
Theorem~\ref{T:ff}. This is interesting because, in a certain precise sense, smooth force-free fields with a
nonconstant factor are rare, as discussed in~\cite{MYZ,ARMA}.

\subsection{Strategy of the proof}

The strategy of the proofs of Theorems~\ref{T:main} and~\ref{T:main2} is similar, so let us focus on the former. The basic
idea behind Theorem~\ref{T:main} is motivated by the work of Bruno and
Laurence~\cite{T09} on MHD equilibria on small
perturbations of an axisymmetric toroidal domain. The
perturbative construction they use in their proof, however, hinges strongly on
having approximately axisymmetric solutions, where one can obtain very
precise information about the solutions and their trajectories, and
cannot be extended to deal with toroidal domains that are not
approximately symmetric.

To explain the gist of our approach, let us stick to
the simplest case,
$N=2$. The case of an arbitrary~$N\geq2$ is obtained by repeating the process $N-1$ times.
Our initial observation (Lemma~\ref{L.Euler})
is that, if we have two Beltrami fields $B_1,B_2$ defined on two disjoint
domains $\Om_1,\Om_2':=\Om_2\backslash\overline{\Om_1}$, with
$\overline{\Om_1}\subset\Om_2$, one can define
a piecewise smooth MHD equilibrium on the domain~$\Om_2$, with a certain
piecewise constant pressure function~$P$, provided that the
difference $|B_1|^2-|B_2|^2$ is constant on~$\pd\Om_1$.

We start by choosing $B_1$ as
a nondegenerate Beltrami field in the toroidal domain
$\Om_1$, so that the analytic surface $\pd\Om_1$ is a Diophantine invariant torus of $B_1$.
To construct a Beltrami field $B_2$ in an exterior neighborhood of
$\pd\Om_1$, we use a version of the Cauchy--Kovalevskaya theorem for the
curl operator~\cite{Annals} (see also Appendix~\ref{Ap1}) with a Cauchy datum given by an analytic
vector field tangent to
$\pd\Om_1$. A key aspect of this theorem is that one can only grant
the existence of a local solution to the equation provided that the
Cauchy datum satisfies an additional constraint.
When one takes this constraint into account, showing that
$|B_1|^2-|B_2|^2$ is constant on~$\pd\Om_1$ becomes equivalent to proving the
existence of an analytic solution to a certain nonlinear Hamilton--Jacobi
equation on~$\TT^2$.

The key difficulty of the problem is that, as the toroidal
domains we consider are far from the axisymmetric case, we cannot
extract from the equations enough information about the trajectories
of the vector fields. The first manifestation of this difficulty is
that we have not found a way of effectively using trajectories to
analyze the aforementioned Hamilton--Jacobi equation. Instead, we have
shown that one can exploit the fact that the restriction
$B_1|_{\pd\Om_1}$ is conjugate to a Diophantine rotation to regard the
equation as a nonlinear perturbation of the cohomological equation
which appears in KAM theory.  With this approach, we eventually
establish the existence of analytic solutions by means of a Newton
quadratic scheme (Theorem~\ref{L:cohom}).

The next step is to show that the resulting field~$B_2$ does in fact
have an invariant torus enclosing a toroidal domain
$\Om_2\supset\overline{\Om_1}$, which permits to make sense of the
basic geometric configuration used to construct the MHD
equilibrium. To this end, we prove that $\pd\Om_1$ is a twist (in the
KAM theoretic sense) invariant torus of $B_2$, so that it is
accumulated by a set of Diophantine analytic invariant tori. However,
the key difficulty is that we cannot compute a good approximation for the
trajectories of~$B_2$. This means that we do not have enough information to
apply the existing KAM theorems for divergence-free fields (see
e.g.~\cite{MW,KKP14,Acta,KKP20}), which are based on studying the Poincar\'e map of the field on a transverse section.

To solve this problem, we establish a KAM theorem for
divergence-free vector fields in $\RR^3$ with two key features
that make it rather different from other KAM theorems in the same context~\cite{Sevr,BHT,KKP14}. First, it applies to vector fields which do not need to be
approximately integrable or in Birkhoff normal form. Secondly, the twist condition is written solely in terms of the vector
field and of the approximate invariant torus. An additional advantage
is that the formulas take a particularly simple form when the field is
Beltrami. Recall that a KAM theorem for perturbations of integrable
volume-preserving diffeomorphisms was obtained in~\cite{CS,X,Y}.

\subsection{Organization of the paper}

After recalling the definition of weak MHD equilibria, in
Section~\ref{S:weak} we prove a lemma ensuring that one can
construct piecewise smooth MHD equilibria by gluing two Beltrami
fields defined on non-intersecting domains with a common boundary
component, provided that the boundary traces of these Beltrami fields
satisfy a certain constraint. The main arguments of the proofs
of Theorems~\ref{T:main} and~\ref{T:main2} are presented in Sections~\ref{S:Tmain} and~\ref{S.freebound}. For
clarity of exposition, however, the two essential technical points of
proof (which are of independent interest) are relegated to Sections~\ref{S.cohom}
and~\ref{sec:teo}. Specifically, in Section~\ref{S.cohom} we solve,
using a cohomological equation, the Hamilton--Jacobi equation
associated with the constraint that we came across in
Section~\ref{S:weak}. Also, in Section~\ref{sec:teo} we prove our
new KAM theorem for divergence-free vector
fields in $\RR^3$. Section~\ref{S:nondeg} is devoted to rigorously proving that
thin toroidal domains of any topology are generically
nondegenerate (of type I and II). The existence result for Lipschitz-continuous force-free fields with a
nonconstant factor is presented in
Section~\ref{S:ff}. The paper concludes with two technical
appendices. In the first appendix we show that Beltrami fields are analytic
up to the boundary if the domain is analytic, and in the second we record
certain results for Beltrami fields that we proved
in~\cite{Annals,Acta,ELP} and which are relevant for the problem under
consideration.
%, and in the third we present some basic
%estimates for the quadratic Newton's scheme that we employ in
%Sections~\ref{S.cohom} and~\ref{sec:teo}.

\section{Construction of weak MHD equilibria from Beltrami fields}\label{S:weak}

In this section we introduce the definition of a weak MHD equilibrium. We say that a pair $(B,P)$ of class, say, $L^2(\Om)$ is a
{\em weak solution to the stationary MHD equations} in $\Om$ if
\[
\int_{\Om} \left[(B\otimes B)\cdot \nabla w- \left(P+\frac12|B|^2\right)\Div w\right]\, dx=0\quad
\text{and}\quad \int_{\Om}B\cdot \nabla\phi\,dx=0
\]
for any vector field $w\in C^1_c(\Om)$ and any scalar function
$\phi\in C^1(\Om)$. Of course, if $B$ and~$P$ are continuously differentiable, this
is equivalent to saying that they satisfy Equations~\eqref{Eq.MHD}--\eqref{eq.MHD3} in~$\Om$. As usual, the subscript $c$ means that we take functions with compact support.

\begin{lemma}\label{L.Euler}
Let $\{\Om_k\}_{k=1}^N$ be $N\geq2$ bounded domains in $\RR^3$ with
smooth connected boundaries. Assume that these domains are nested in
the sense that $\overline{\Om_{k-1}}\subset \Om_k$ for all $1\leq
k\leq N$, with $\Om_0:=\emptyset$. With
$\Om_k':=\Om_k\backslash\overline{\Om_{k-1}}$, suppose furthermore
that the vector field $B_k$ satisfies the equation $\curl B_k=\la_kB_k$ in~$\Om_k'$ for some nonzero constant $\la_k$. Assume that $B_k$ is tangent to the boundary of $\Om_k'$ and that
\begin{equation}\label{Eq.weak}
|B_{k+1}|^2-|B_k|^2=2c_k \text{ on } \pd\Om_k
\end{equation}
for all $1\leq k\leq N-1$, where $c_k$ are constants. Then
\[
  B(x):=\sum_{k=1}^N B_k(x)\, 1_{\Om_k'}(x)
\]
is a piecewise smooth MHD equilibrium on~$\Om_N$
with piecewise constant pressure
\[
  P(x):=c_0 -\sum_{k=1}^{N-1}c_k\, 1_{\Om_N\backslash\overline{\Om_k}}(x)\,.
\]
Here $c_0$ is any real constant (in particular, it can be chosen so that $P(x)=0$ if $x\in\pd\Om_N$).
\end{lemma}
\begin{remark}
As $B_k$ is a Beltrami field defined in smooth domains and tangent to
the boundary, it is standard~\cite{BS} that $B_k$ is smooth up to the
boundary. Therefore, the constraint~\eqref{Eq.weak} makes sense
pointwise. A related analytic regularity result up to the boundary,
which will be needed later on, is proved in
Appendix~\ref{A.analytic}.
\end{remark}

\begin{remark}
The result and the proof remain valid when $\la_k=0$ if the
corresponding vector field $B_k$ is additionally assumed to be
divergence-free in $\Om_k'$.
\end{remark}

\begin{proof}
To keep the formulas as simple as possible, we will prove the result
for $N=2$; the general case is analogous. We start by noticing that, for all $\phi\in C^1(\Om_2)$,
\begin{align*}
\int_{\Om_2}B\cdot \nabla \phi\, dx&=\int_{\Om_1} B_1\cdot \nabla \phi\,
                                       dx+\int_{\Om_2'} B_2\cdot \nabla \phi\,
                                       dx \\
&= \int_{\pd \Om_1} \phi\, B_1\cdot N\, dS+\int_{\pd (\Om_2')} \phi\, B_2\cdot N\, dS=0\,,
\end{align*}
where we have used that $\Div B_1=\Div B_2=0$ in their respective domains and $B_1\cdot N=0$ on~$\pd\Om_1$ and $B_2\cdot N=0$ on~$\partial\Om_2'$. Hence $\Div B=0$ in the sense of distributions.

Let us now take an arbitrary vector field $w\in
C^1_c(\Om_2)$. We can write
\begin{align*}
 I:=&\int_{\Om_2} \left[ (B\otimes B)\cdot \nabla w- \left(P+\frac12|B|^2\right)\Div w\right]\, dx\\
 &=\int_{\Om_1} \left[ (B_1\otimes B_1)\cdot \nabla w- \left(c_0+\frac12|B_1|^2\right)\Div w\right]\, dx\\
 &\qquad +\int_{\Om_2'} \left[ (B_2\otimes B_2)\cdot \nabla w- \left(c_0-c_1+\frac12|B_2|^2\right)\Div w\right]\, dx=:I_1+I_2
\end{align*}
Integrating by parts, and using that $B_1\cdot N=0$ on $\pd\Om_1$, we easily obtain
\begin{align*}
I_1&=\int_{\Om_1}\left[(B_1\otimes B_1)\cdot\nabla w + \frac12\nabla|B_1|^2\cdot w\right]\,dx-\int_{\pd\Om_1}\left (c_0+\frac12|B_1|^2\right)w\cdot N\,dS\\
&=-\int_{\Om_1}\left[\Div(B_1\otimes B_1)-\frac12\nabla|B_1|^2\right]\cdot w\,dx-\int_{\pd\Om_1}\left (c_0+\frac12|B_1|^2\right)w\cdot N\,dS\\
&=-\int_{\pd\Om_1}\left (c_0+\frac12|B_1|^2\right)w\cdot N\,dS\,.
\end{align*}
To pass to the last equation we have used the well known identity for Beltrami fields
\[
\Div(B_1\otimes B_1)=\frac12\nabla|B_1|^2\,.
\]

Analogously, using the same identity for $B_2$ on $\Om_2'$ and that
$B_2\cdot N=0$ on the boundary, we can compute the term $I_2$. Notice
that by the connectedness of the boundaries of $\Om_1$ and $\Om_2$ we
have $\partial\Om_2'=\pd\Om_1\cup\pd\Om_2$, and that the outward
pointing normal vector of $\Om_2'$ on $\pd\Om_1$ is $-N$. We thus obtain
\begin{align*}
I_2&=\int_{\Om_2'}\Big[(B_2\otimes B_2)\cdot\nabla w + \frac12\nabla|B_2|^2\cdot w\Big]\,dx+\int_{\pd\Om_1}\Big(c_0-c_1+\frac12|B_2|^2\Big)w\cdot N\,dS\\
&=\int_{\pd\Om_1}\Big(c_0-c_1+\frac12|B_2|^2\Big)w\cdot N\,dS\,.
\end{align*}
The surface integral is taken only on $\pd\Om_1$ because $w=0$ on $\pd\Om_2$.

Putting together these computations and using the boundary condition~\eqref{Eq.weak}, we finally conclude that
\[
I=\int_{\pd\Om_1}\left[\frac12\left(|B_2|^2-|B_1|^2\right)-c_1\right]w\cdot N\,dS=0\,,
\]
for all $w\in C^1_c(\Om_2)$. It then follows
that $(B,P)$ is a weak solution of the MHD equations in $\Om_2$, as claimed.
\end{proof}

\begin{remark}\label{R:current}
It is easy to check that the plasma current $J:=\curl B$ of the
solution constructed in Lemma~\ref{L.Euler} is the vector-valued distribution
\[
  J =\sum_{k=1}^N\la_k
  B_k\, 1_{\Om_k'}+\sum_{k=1}^{N-1}(B_k-B_{k+1})\times N_k\, dS_k\,,
\]
where $dS_k$ and $N_k$ are the area measure and outer unit normal on
$\pd\Om_k$. The current sheet terms appearing in this formula are a
consequence of the discontinuity of the magnetic field across the
surfaces $\pd\Om_k$.
\end{remark}

\section{Fixed boundary equilibria: proof of Theorem~\ref{T:main}}\label{S:Tmain}

In this section we show how to implement the strategy discussed in the
Introduction to prove the first main result of this paper, modulo some
technical results that will be presented in later sections.

\subsection{Nondegenerate toroidal domains of type I}

Let us begin by defining the class of toroidal domains that we
consider in Theorem~\ref{T:main}.

\begin{definition}\label{D:torus}
A toroidal domain~$\Om\subset\RR^3$ is {\em
nondegenerate of type I}\/ if there exists a Beltrami field $B$
on~$\Om$ such that the following conditions hold:
\begin{enumerate}
\item $\pd\Om$ is a Diophantine invariant torus of the
  field.

  \item The $2\times 2$ constant matrix
\begin{equation}\label{eq.gen}
M:=\int_{\TT^2}G(\vp)^{-1}\,\left(
    \begin{array}{cc}
      1-\om_1\partial_{\vp_1}\mathcal R & -\om_2\partial_{\vp_1}\mathcal R \\
      -\om_1\partial_{\vp_2}\mathcal R & 1-\om_2\partial_{\vp_2}\mathcal R \\
    \end{array}
  \right)\,d\vp
\end{equation}
is invertible. Here $\om$ is the frequency vector of $B$ on $\pd\Om$,
$\vp$ are the linearizing coordinates in Equation~\eqref{eq.torus},
$G$ is the metric matrix (or first fundamental form) of the surface
$\pd\Om$ in the coordinates~$\vp$, and
$\mathcal R(\vp)$ is the unique zero mean solution to
the equation on the torus
\[
\om_1\pd_{\vp_1}\mathcal R+\om_2\pd_{\vp_2}\mathcal R=\varkappa-|X|^2\,,
\]
where $X:= B|_{\pd\Om}$ and the constant~$\varkappa$ is chosen so that
$\int_{\TT^2}(\varkappa-|X|^2)\,d\vp=0$. By rescaling the Beltrami field, we can henceforth assume that the
above constant is $\varkappa=1$ (which has already been considered when writing Equation~\eqref{eq.gen}).
\end{enumerate}
\end{definition}

The Beltrami field and the analytic toroidal
domain satisfying the nondegeneracy assumption are $B_1$ and $\Om_1$,
and the corresponding eigenvalue is~$\la_1$. Clearly, the nondegeneracy of $B_1$ is invariant under rescaling the field by a nonzero constant.

Here and in what follows,
\[
  [f]_{\TT^2}:=
  \frac1{4\pi^2}\int_{\TT^2} f\, d\vp
\]
denotes the average of a function
on~$\TT^2$ (in $\vp$-coordinates) and we set $\om^\perp:=
(\om_2,-\om_1)$ for each two-component vector $\om
=(\om_1,\om_2)$.

%\begin{remark}
%  As shall see later on that hypothesis~(iii) is not used in the proof
%  of Theorem~\ref{T:main} (it does play an important role in the proof
%  of Theorem~\ref{T:main2}, though). In the result about the existence
%  of Lipschitz Beltrami fields with nonconstant factor presented in Theorem~\ref{T:ff}, only the hypothesis~(i) will be
%  needed. Consequently, we could slightly relax the definition of a
%  nondegenerate toroidal domain depending on whether we are
%  considering fixed boundary MHD equilibria, free boundary equilibria
%  or Beltrami fields with a nonconstant factor.
%\end{remark}

\subsection{Construction of the first layer}\label{S.pt}

As $\pd\Om_1$ is an invariant torus of $B_1$ with Diophantine frequency
vector $\om^{(1)}$, it is standard~\cite{LGJV} that one can parametrize the
invariant torus by an embedding $K_1 : \TT^2 \rightarrow \RR^3$
satisfying the equation
\[
L_{\om^{(1)}} K_1 = B_1 \circ K_1\,,
\]
where $L_{\om^{(1)}}K_1:=DK_1\om^{(1)}$ is the pointwise derivative
of~$K_1$ in the direction of~$\om^{(1)}$. In this picture, of course,
the invariant torus is the image $\pd\Om_1=K_1(\TT^2)$. Let us
emphasize from the beginning that parametrizing the invariant tori by
embeddings is essential for the KAM theorem that we will prove in
Section~\ref{sec:teo} (Theorem~\ref{teo:kam:div}), which will play a key role in this proof
later on.

Since the boundary of the domain $\Om_1$ is analytic,
Theorem~\ref{T.analytic} implies that $B_1$ is analytic up to the
boundary. A theorem of Herman~\cite{Y,Yoccoz} then ensures that the linearizing parametrization $K_1$ is analytic, or in other words, there
are analytic coordinates $\vp:\pd\Om_1\to\TT^2$ in which $X_1:=B_1|_{\pd\Om_1}$ takes the form
\[
X_1=G_{K_1}^{-1}DK_1^\top B_1 \circ K_1 = \om_1^{(1)}\pd_{\vp_1}+\om_2^{(1)}\pd_{\vp_2}\,.
\]
Here  $G_{K_1}:=DK_1^\top DK_1$ is the matrix representation of the
pullback of the Euclidean metric to the surface $\pd\Om_1$ obtained
using the embedding $K_1$, so it is a positive definite $2\times 2$ symmetric
(nonconstant) matrix.

Theorem~\ref{L:cohom} implies that, for any constant~$c_2$ that is
small enough in absolute value, there is an analytic vector field
$X_2$ on $\pd\Om_1$ such that:
\begin{enumerate}
\item The $1$-form that is dual to~$X_2$ with respect to the
  metric on~$\pd\Om_1$ induced by the Euclidean metric in~$\RR^3$, is closed.
  \item The pointwise norm of~$X_2$, computed with the induced
    metric on~$\pd\Om_1$, satisfies the equation
    \[
|X_2|^2=(1+b_2)|X_1|^2+c_2
\]
for some constant bounded as $|b_2|\leq C|c_2|$.
\item The vector field $X_2$ depends continuously on the parameter
  $c_2$, in the $C^r$-topology of vector fields for any~$r$, and is Diophantine with frequency vector
\begin{equation}\label{eq.om2}
\om^{(2)}:=(1+c_2)^{1/2}\om^{(1)}\,.
\end{equation}
\end{enumerate}
In particular, one can write
\[
X_2=X_1+\mathcal O(c_2)=\om_1^{(1)}\pd_{\vp_1}+\om_2^{(1)}\pd_{\vp_2}+\cO(c_2)\,,
\]
where in what follows $\cO(c_2)$ stands for a quantity (which may vary
from line to line) whose $C^r$ norm is bounded by $C|c_2|$, for any fixed integer $r$. We remark that the fact that the frequency vectors $\om^{(2)}$ and $\om^{(1)}$ are proportional, cf. Equation~\eqref{eq.om2}, is a consequence of the proof of Theorem~\ref{L:cohom}. It is not a necessary condition, but it is convenient to write all the estimates of the proof of Lemma~\ref{L:Newton} in a simple way, and to keep the same Diophantine constants.

Now we consider the Cauchy problem
\begin{equation*}
\curl B_2'=\la_2 B_2'\,, \qquad B_2'|_{\pd\Om_1}=X_2\,,
\end{equation*}
for some nonzero constant $\la_2\neq \la_1$ that we will fix later. Since
$\pd\Om_1$ and $X_2$ are analytic, and the $1$-form dual to $X_2$
is closed, Theorem~\ref{T:CK} (which is a sort of Cauchy--Kovalevskaya
theorem for the curl operator) implies that
there exists a unique analytic solution to this Cauchy problem in a
neighborhood of $\pd\Om_1$. Eventually, we will only be interested in the
behavior of the solution outside ${\Om_1}$.

By construction, $\pd\Om_1$ is a Diophantine invariant torus of the
vector field $B_2'$. We claim that it is twist (in the KAM sense; see
Definition~\ref{def:ndeg}) for almost all choices of the constant
$\la_2$. This implies that $\pd\Om_1$ is accumulated by a set of
Diophantine invariant tori of $B_2'$ contained in $\RR^3\backslash
\overline{\Om_1}$. Since $B_2'$ is analytic, these tori are analytic
as well. This follows from Corollary~\ref{Cor_KAM} to
the KAM theorem for
divergence-free fields that we shall prove in Section~\ref{sec:teo}.

Let us denote by $K_2:\TT^2\to\RR^3$ an embedding that is a linearizing parametrization of the
invariant torus $\pd\Om_1$ with frequency vector $\om^{(2)}$ of the
vector field $B_2'$. Then, we can introduce coordinates
(which we still denote by~$\vp$) such that $X_2=B_2'|_{\pd\Om_1}$ becomes the Diophantine linear field
\[
X_2=G_{K_2}^{-1}DK_2^\top B_2'\circ K_2=\om_1^{(2)}\pd_{\vp_1}+\om_2^{(2)}\pd_{\vp_2}\,.
\]
In general, $K_2$ is different from the parametrization $K_1$ that
linearizes $X_1$, but it follows from the previous discussion that
both parametrizations differ by a higher order correction, i.e.,
\begin{equation}\label{eq.K2}
K_2=K_1+\cO(c_2)\,.
\end{equation}

Since $B_2'$ satisfies the equation $\curl B_2'=\la_2B_2'$, an easy computation shows that the following identity holds:

\begin{lemma}
$DB_2'^\top+DB_2'=2DB_2'^\top+\la_2 B_2'\,\times$, where $DB_2'$ is the Jacobian matrix of $B_2'$ and $\times$ denotes the vector product, both computed in Cartesian coordinates.
\end{lemma}
\begin{proof}
The proof is straightforward. Indeed,
$$DB_2'^\top+DB_2'=2DB_2'^\top + (DB_2'-DB_2'^\top)=2DB_2'^\top + \curl B_2'\,\times\,,$$
so the claim follows from the equation $\curl B_2'=\la_2B_2'$.
\end{proof}

To invoke Corollary~\ref{Cor_KAM}, we must check the twist condition
(cf.\ Definition~\ref{def:ndeg}). This involves computing a
two-component vector field (or $2\times1$ matrix) appearing in
Equation~\eqref{eq:condA}, which in this case takes the form
\begin{equation*}
A_2=
 - \frac{G_{K_2}^{-1}}{|n_2|^2}
\Big[2DK_2^\top DB_2'^\top n_2+\la_2 DK_2^\top (B_2'\times n_2)\Big]\,.
\end{equation*}
Here $DB_2'^\top$ and $B_2'$ are evaluated at $K_2(\vp)$ and the
normal vector
\[
n_2(\vp):=\partial_{\vp_1}K_2(\vp) \times \partial_{\vp_2}K_2(\vp)
\]
is defined in terms of $K_2$ as in Definition~\ref{D:torus}. Observe that $DK_2$ is a $3\times 2$ matrix.

Since the vector field $B_2'\times n_2$ is tangent to $\pd\Om_1$ and
perpendicular to $B_2'$, we infer that there is a nonvanishing vector (given by a $2\times 1$ matrix) $\alpha_2$ on $\pd\Om_1$ such that
\[
B_2'\times n_2=DK_2\,\alpha_2\,.
\]
Therefore,
\[
  DK_2^\top (B_2'\times n_2)=(DK_2^\top DK_2)
  \alpha_2=G_{K_2}\alpha_2\,.
\]
The matrix $A_2$ then takes the form
\begin{align*}
A_2&=- \frac{2G_{K_2}^{-1}}{|n_2|^2}DK_2^\top DB_2'^\top n_2-\frac{\la_2}{|n_2|^2}\,\alpha_2\\
&=\frac{2G_{K_2}^{-1}}{|n_2|^2}DK_2^\top L_{\om^{(2)}}n_2-\frac{\la_2}{|n_2|^2}\,\alpha_2=:A^{(1)}_2+A^{(2)}_2\,,
\end{align*}
where we have used Lemma~\ref{lem:trace} to pass to the second
equality. It is clear from this expression that the vector $A_2(\vp)$
only depends on the way the torus $\pd\Om_1$ is embedded in $\RR^3$,
on the Diophantine vector $\om^{(2)}$, on the parametrization $K_2$
linearizing $X_2$, and on the eigenvalue $\la_2$.

According to Definition~\ref{def:ndeg}, the invariant torus $\pd\Om_1$
of $B_2'$ is {\em twist}\/ if the twist constant
\begin{align}
T_2:&=\big([A^{(1)}_2]_{\TT^2}+ [A^{(2)}_2]_{\TT^2}\big)\cdot (\om^{(2)})^\perp\notag\\
&=[A^{(1)}_2]_{\TT^2}\cdot (\om^{(2)})^\perp - \la_2|\om^{(2)}|\bigg[\frac{F_2}{|n_2|^2}\bigg]_{\TT^2}\label{T2}
\end{align}
is nonzero. Here  $ F_2(\vp)$ is the function defined as the projection of the field $\alpha_2$ in the $(\om^{(2)})^\perp$ direction, i.e.,
\[
F_2:=\alpha_2 \cdot \frac{(\om^{(2)})^\perp}{|\om^{(2)}|}\,.
\]
This function is nonvanishing because the field $DK_2\,\alpha_2$ is perpendicular to $B_2'|_{\pd\Om_1}$, so $\alpha_2$ and $\om^{(2)}$ cannot be proportional at some point of $\TT^2$.

Arguing in the same way, we obtain an analogous expression for the
twist constant $T_1$ of the invariant torus $\pd\Om_1$ for the vector
field $B_1$. As $\pd\Om_1$ is an invariant torus for $B_1$ and $B_2'$,
and as the corresponding parametrizations $K_1$ and $K_2$ and Diophantine vectors $\om^{(1)}$ and $\om^{(2)}$ differ just by an error of order $c_2$ by Equations~\eqref{eq.om2}-\eqref{eq.K2}, we conclude that
\begin{align*}
T_2&=T_1-\la_2\Big(|\om^{(1)}|\bigg[\frac{F_1}{|n_1|^2}\bigg]_{\TT^2}+O(c_2)\Big)+\la_1|\om^{(1)}|\bigg[\frac{F_1}{|n_1|^2}\bigg]_{\TT^2}+O(c_2)\\
&=: T_1-a\la_2+b\,,
\end{align*}
where $n_1:=\partial_{\vp_1}K_1 \times
  \partial_{\vp_2}K_1$ and $F_1$ is also nonvanishing. The constants $a,b$ are therefore nonzero if
$|c_2|$ is small enough, so $T_2\neq 0$ provided that
\begin{equation}\label{lambda}
\la_2\neq \frac{b+T_1}{a}=1+\frac{T_1}{|\om^{(1)}|\Big[\frac{F_1}{|n|^2}\Big]_{\TT^2}}+O(c_2)\,.
\end{equation}
This shows that $\pd\Om_1$ is a twist invariant torus of $B_2'$ for almost all choices of $\la_2$.

Hence, we can take a Diophantine analytic invariant torus $\Si_2$ of the
vector field $B_2'$, lying outside $\overline{\Om_1}$, which is
$\eta$-close to $\pd\Om_1$. By this we mean that, for any fixed~$r$, there is a
diffeomorphism $\Psi_1$ of~$\RR^3$ which maps $\pd\Om_1$ into
$\Si_2$ and which is close to the identity as $\|\Psi_1-\id\|_{C^r}<\eta$.
The invariant torus~$\Si_2$ is then the boundary of a toroidal domain
$\Om_2\supset\overline{\Om_1}$. It is easy to check that the matrix $M_2$ in Equation~\eqref{eq.gen}, associated to the vector field $B_2'|_{\pd\Om_2}$ is related to the matrix $M_1$ of $B_1|_{\pd\Om_1}$ as
\[
M_2=M_1+O(\eta+|c_2|)\,.
\]
As $M_1$ is invertible and $\pd\Om_1$ is accumulated by Diophantine invariant tori of $B_2'$,
we can therefore take $\eta$ (and $|c_2|$) small enough so that $M_2$ is invertible
too. We then conclude that $\Om_2$ is a nondegenerate toroidal
domain of type I.

\subsection{Conclusion of the proof}

As $\Om_2$ is another nondegenerate toroidal domain of type I,  we can repeat
  the argument to construct a vector field $B_3'$ in a neighborhood of $\pd\Om_2$ that solves the equation
\[
\curl B_3'=\la_3 B_3'\,, \qquad B_3'|_{\pd\Om_2}=X_3\,,
\]
for some constant $\la_3\neq \la_2$, and the Cauchy datum $X_3$ satisfies
\[
|X_3|^2=(1+b_3)|\widetilde X_2|^2+c_3
\]
with arbitrarily small constants $c_3$ and $b_3=O(c_3)$. Here
$\widetilde X_2:=B_2'|_{\pd\Om_2}$. Again, one can pick $c_3$ and
$\la_3$ so that $\pd\Om_2$ is a twist Diophantine invariant torus of
$B_3'$. Therefore there is an analytic Diophantine invariant torus of
$B_3'$, which is the boundary of another nondegenerate toroidal domain of type I $\Om_3\supset
\overline{\Om_2}$.

This process can be iterated $N-1$ times to obtain a family of
(analytic) nested tubes $\{\Om_k\}_{k=1}^N$, different constants
$\la_k$, small constants $c_k,b_k$ and vector fields $B_k'$ satisfying
$\curl B_k'=\la_kB_k'$ in $\Om_k':=\Om_k\backslash
\overline{\Om_{k-1}}$ for all $2\leq k\leq N$.

To construct a weak solution $(B,P)$ of the MHD equations in the
toroidal domain $\Om_N$, we set
\begin{align*}
B(x)&:=B_1(x)\,1_{\Om_1}(x)+ \sum_{k=2}^N B_k'(x)\, 1_{\Om_k'}(x)\prod_{j=2}^k(1+b_j)^{-1/2}\\
  P(x)&:=  p_1 \,1_{\Om_1}(x)+ \sum_{k=2}^N p_k \, 1_{\Om_k'}(x)\,.
\end{align*}
The constant $p_1$ is arbitrary, and the constants $p_k$ are defined in terms of $c_j,b_j$ as
\[
p_k:=p_1-\frac12\sum_{l=2}^{k}\prod_{j=2}^l(1+b_j)^{-1}c_l\,,
\]
Note that, generically, $p_k\neq p_j$ if $k\neq j$.
A straightforward application of Lemma~\ref{L.Euler} shows that
$(B,P)$ is a piecewise smooth MHD equilibrium with all the properties
stated in Theorem~\ref{T:main}.

\begin{remark}
In view of potential applications of our method of proof in the context of plasma physics, we want to highlight that the toroidal surfaces $\pd\Om_2,\dots,\pd\Om_N$ we construct in the proof of Theorem~\ref{T:main} are not too distorted from one to the other, and their separation is small. Quantifying such distortions is feasible (but not trivial) using quantitative versions of the Cauchy-Kovalevskaya and the KAM theorems (see also Remark~\ref{R.quant}).
\end{remark}

\section{Free boundary equilibria: proof of
  Theorem~\ref{T:main2}}\label{S.freebound}

We first introduce the class of toroidal domains that we
consider in Theorem~\ref{T:main2}.

\begin{definition}\label{D:torusII}
A toroidal domain~$\Om\subset\RR^3$ is {\em
nondegenerate of type II}\/ if there exists a Beltrami field $B$
on~$\Om$ (with eigenvalue $\la$) such that the following conditions hold:
\begin{enumerate}
\item $\pd\Om$ is an invariant torus of the
  field with Diophantine frequency vector $\om$.

\item The {\em twist constant}\/ $T$ (see Definition~\ref{def:ndeg}) satisfies
  \[
T+\la\,\bigg[\frac{\al\cdot \om^\perp}{|n|^2}\bigg]_{\TT^2}\neq0\,,
  \]
  where $\al\cdot\om^\perp\equiv \al_1\om_2-\al_2\om_1$, $K:\TT^2\to\RR^3$ is the linearizing embedding of~$\pd\Om$ in
  coordinates~$\vp$ (so, in particular, $\pd\Om=K(\TT^2)$),
  $n:=\pd_{\vp_1}K\times \pd_{\vp_2}K$ is a normal vector and the $\RR^2$-valued
  function $\al(\vp)$ is defined as
  \[
    (B\circ K)\times n=: DK\, \al\,.
  \]
\end{enumerate}
The nondegeneracy of $B$ is invariant under rescaling the field by a nonzero constant.
\end{definition}

The proof of Theorem~\ref{T:main2} follows the same strategy as the proof of
Theorem~\ref{T:main}, although it is easier because it does
not make use of the Hamilton--Jacobi equation we study in
Section~\ref{S.cohom}. Therefore, to make the presentation easier, we will use
the same notation as in the previous section without further mention.

As the analytic toroidal domain is nondegenerate in the sense of
Definition~\ref{D:torusII}, there exists a Beltrami field $B$ in $\Om$, satisfying the equation
\[
\curl B =\la B
\]
for some nonzero constant $\la$ and the boundary condition $B\cdot N=0$, such that $\pd\Om$ is a Diophantine
invariant torus of $B$ with frequency vector $\om$ and its twist
constant satisfies
\begin{equation}\label{nondegh}
T+\la\bigg[\frac{\al\cdot \om^\perp}{|n|^2}\bigg]_{\TT^2}\neq 0\,.
\end{equation}

We focus on the special case of a globally force-free equilibrium with $P=0$, whose magnetic field $B$ is a Beltrami field. In Remark~\ref{R.freeb} we explain how to modify this construction to obtain more general equilibria. It is then clear that
$(B,B\ext,\Om)$ are a solution to the equations for a free
boundary MHD equilibrium with external current $J\ext$ if the external
magnetic field and the external current satisfy
\begin{subequations}\label{external}
\begin{align}
\curl B\ext &=J\ext  \!\quad \text{  in } \RR^3\backslash \overline{\Om}\,,\label{Bext1}\\
\Div B\ext &=0 \qquad \text{  in } \RR^3\backslash \overline{\Om}\,,\\
  B\ext \cdot N&=0 \qquad \text{  on } \pd\Om\,, \label{Bext3}\\
|B\ext |^2-|B|^2&=0 \qquad \text{  on } \pd\Om \label{Bext4}\,.
\end{align}
To ensure that $J\ext$ is a current sheet, we also aim to construct a
toroidal domain $\Om'\supset \BOm$ and a tangent vector field $J$
on~$\pd\Om'$ such that
\begin{align}
J\ext &=J\, dS \,,\\
 B\ext &=0 \qquad \text{  in } \RR^3\backslash \overline{\Om'}\,,\label{Bext6}
\end{align}
\end{subequations}
where $dS$ is the surface measure on $\pd\Om'$.
Note that the tangent vector field  $J$ must be divergence-free with respect to the induced
metric on~$\pd\Om'$ because Equation~\eqref{Bext1} implies that, in the sense of distributions,
\[
0=  \Div
  J\ext=(\Div_{\pd\Om'}J)\, dS\,.
\]

Thus proving Theorem~\ref{T:main2} boils down to constructing a domain
$\Om'$ and an analytic divergence-free tangent vector field~$J$
on~$\pd\Om'$ such that the solution to the exterior div-curl
problem~\eqref{Bext1}--\eqref{Bext3} on $\RR^3\backslash\BOm$
satisfies the conditions~\eqref{Bext4}--\eqref{Bext6}.

To construct solutions to this overdetermined system, we follow
the same philosophy as in Section~\ref{S:Tmain}. Let $X:=B|_{\pd\Om}$
be the restriction of the Beltrami field $B$ on the boundary of the
domain~$\Om$.  Observe that $X$ is analytic in view of
Theorem~\ref{T.analytic} and the associated $1$-form $X^\flat$ is closed on
$\pd\Om$ by Theorem~\ref{T:CK}. Therefore, the Cauchy problem
\[
\curl h =0\,, \qquad \Div h=0\,, \qquad h|_{\pd\Om}=X\,,
\]
has a unique analytic solution in a small tubular neighborhood~$U$ of $\pd\Om$
as a consequence of Theorem~\ref{T:CK}. By construction, $h\cdot N=0$
and $|h|^2=|B|^2$ on $\pd\Om$. Lemma~\ref{L.Euler}
then ensures that the field
\[
B\, 1_\Om + h \, 1_{U'}\,,
\]
with $U':= U\backslash\BOm$,
is a weak solution to
the stationary MHD equations in the toroidal domain $\Om_2:=U\cup \Om$.

Proceeding just as in Subsection~\ref{S.pt}, the twist constant $T_h$
of the invariant torus $\pd\Om$ of the harmonic field $h$ can
be readily shown to be
\[
T_h=T+\la |\om|\bigg[\frac{F}{|n|^2}\bigg]_{\TT^2}= T+\la \bigg[\frac{\al\cdot\om^\perp}{|n|^2}\bigg]_{\TT^2}\,,
\]
where $T$ is the twist constant of~$B$ (cf. Definition~\ref{def:ndeg}).
This is simply Equation~\eqref{T2}, where we have set $\la_2=0$
because the field~$h$ is harmonic and $c_2=0$ because $|h|^2=|B|^2$ on $\pd\Om$.

The nondegeneracy assumption of type II, i.e., Equation~\eqref{nondegh}, ensures that $T_h\neq 0$.
Thus Corollary~\ref{Cor_KAM} implies that $\pd\Om$ is accumulated (in
both components of its complement) by analytic Diophantine invariant
tori of $h$. We can therefore choose an analytic domain $\Om'\supset\BOm$
whose boundary is one of these invariant tori.

To conclude, let us now define the vector field
$$
B\ext(x) :=h(x)\,1_{\Om'\backslash \BOm}(x)
$$
for $x\in \RR^3\backslash \BOm$. As $h$ is divergence free in $\Om'\backslash \BOm$ and tangent to
$\pd\Om$ and $\pd\Om'$, an elementary computation shows that $\Div
B\ext =0$ in the sense of distributions. Furthermore, the
corresponding current is readily computed using that
$$
\int_{\RR^3\backslash \BOm} B\ext\cdot\curl v\,dx  =\int_{\pd\Om'} v\cdot J\, dS
$$
for any $v\in C^\infty_c(\RR^3\backslash \BOm,\RR^3)$,
where
\[
J:=  h\times N'
\]
and $N'$ is the outer unit normal on
$\pd\Om'$. Therefore, $(B\ext,J,\Om')$ satisfies the system~\eqref{external},
so Theorem~\ref{T:main2} follows.

\begin{remark}\label{R.quant}
Quantitative versions of the Cauchy--Kovalevskaya theorem (Theorem~\ref{T:CK}) and the KAM
theorem (Theorem~\ref{teo:kam:div}) provide an estimate for the
separation that we can obtain between the current sheet $\pd\Om'$ and
the domain $\Om$ in terms of the Diophantine constants of the
frequency vector~$\om$ and of the analyticity radii and the analytic
norms of the different objects that appear in the construction
(namely, the Beltrami field $B|_{\pd\Om}$ and the linearizing
embedding~$K$ of the invariant torus~$\pd\Om$). However, this is a hard job that will be the object of future contributions.
\end{remark}

\section{Solving a Hamilton--Jacobi problem via the cohomological equation}\label{S.cohom}

Let $\Om$ be an analytic nondegenerate toroidal domain of type I in $\RR^3$. By definition, there exists a Beltrami field $B$ in $\Om$ that satisfies the equation
\[
\curl B=\la B
\]
for some constant $\la$, $\partial \Om$ is a Diophantine invariant
torus of $B$, and the corresponding matrix $M$ defined in Equation~\eqref{eq.gen} is invertible. Arguing as in the beginning of Section~\ref{S.pt} we infer that $B$ is analytic up to the boundary and there are analytic coordinates $\vp:\partial\Om\to \TT^2$ such that
\[
Y:=B|_{\pd\Om}=\om_1\pd_{\vp_1}+\om_2\pd_{\vp_2}\,,
\]
where $\om\in\RR^2$ is a Diophantine frequency vector. With a slight
abuse of notation, in this section we will use the same name for a
quantity on $\pd\Om$ (a function or a vector field) and for its
expression in these coordinates.

In this section, if $Z_1,Z_2$ are two vector fields on $\TT^2$, $|Z_1|$ and $Z_1\cdot Z_2$
denote the norm of $Z_1$ and the scalar product of $Z_1$ and $Z_2$, respectively, computed with respect to the metric on
$\pd\Om$ induced by the Euclidean metric, which we write in the
coordinates~$\vp$. As before, $[f]_{\TT^2}$ will denote the mean of a function $f$ in $\vp$-coordinates, i.e.,
\[
[f]_{\TT^2}:=\frac{1}{4\pi^2}\int_{\TT^2}f(\vp)\,d\vp
\]
with $d\vp:=d\vp_1\,d\vp_2$.

The main result of this section is the following theorem. This result
is used in the proof of Theorem~\ref{T:main} to construct
Cauchy data which satisfy the constraint equation~\eqref{Eq.weak}
appearing in the Cauchy--Kovalevskaya theorem for
the curl operator.

\begin{theorem}\label{L:cohom}
  Let $\Om$ and $Y$ be as above. Then, for any constant~$c$ with small
  enough absolute value, there is a nonnegative constant~$b$ and an analytic vector field $X$ on
  $\pd\Om$ of the form
  \[
X=(1+c)^{1/2}Y+\nabla H+a_1\nabla\vp_1+a_2\nabla\vp_2\,,
  \]
  where $\nabla$ denotes the gradient operator on~$\pd\Om$ associated
  with the induced metric,  such that:
\begin{enumerate}
\item $|X|^2=(1+b)|Y|^2+c$.

  \item $X$ is analytically conjugate to a linear field with Diophantine
  frequency vector $\om'=(1+c)^{1/2}\om$.

  \item For any fixed~$r$, the scalar analytic function~$H$ on~$\pd\Om$ and the    constants $b,a_1,a_2$ are bounded as
    \[
\|H\|_{C^r(\pd\Om)}+|b|+|a_1|+|a_2|\leq C|c|\,.
    \]
  \end{enumerate}
  Moreover, $X$ depends continuously on the parameter $c$ in
  the $C^r$-topology of vector fields, for any fixed~$r$.
\end{theorem}

\begin{remark}
We are working in the analytic category because we need analytic solutions to
apply the Cauchy--Kovalevskaya theorem in the proof of
Theorem~\ref{T:main}. In this category, we can prove
Theorem~\ref{L:cohom} using a quadratic Newton scheme. Using
instead a Nash-Moser iteration, one can prove a completely analogous result in the $C^r$ setting, for large enough $r$.
\end{remark}

\begin{remark}
The expression of~$X$ guarantees that the dual $1$-form of $X$ (computed with respect to the metric
  induced from $\RR^3$), which we denote by $X^\flat$, is closed: $d
  X^\flat=0$. This condition is also required to apply the
  Cauchy--Kovalevskaya theorem. We also observe that the introduction of nonzero constants $a_1,a_2$ in the expression of $X$ is important to prove Lemma~\ref{L:Newton} below, where they are crucial to ensure the solvability of certain cohomological equations that appear in the Newton's iterative scheme.
\end{remark}

\begin{proof}
As in Section~\ref{S:Tmain}, we can rescale~$B$  so that
\begin{equation}\label{eqnorm}
[|Y|^2]_{\TT^2}=1\,.
\end{equation}

In coordinates $\vp$, condition (i) for the vector field $X$ is
equivalent to picking the analytic function $H:\TT^2\to\RR$ and the
constants $a=(a_1,a_2)$ so that the equation
\begin{equation}\label{eqHJ}
2(1+c)^{1/2}L_\om H+|\nabla H|^2+2\nabla H\cdot (a\nabla\vp)+f=b|Y|^2+c(1-|Y|^2)
\end{equation}
is satisfied. Here $L_\om\equiv
\om_1\partial_{\vp_1}+\om_2\partial_{\vp_2}$ is the derivative in the
$\om$~direction and we have set
\begin{align*}
f&:=(a\nabla\vp)^2+2(1+c)^{1/2}(a\om)\,,\\
b&:=\Big[|\nabla H|^2+2\nabla H\cdot (a\nabla\vp)+f\Big]_{\TT^2}\,.
\end{align*}
Throughout this proof, we use
the shorthand notation
\begin{align*}
a\nabla \vp&:= a_1\nabla \vp_1+a_2\nabla \vp_2\,,\\
a\om&:=a_1\om_1+a_2\om_2\,.
\end{align*}

To study the existence of solutions to Equation~\eqref{eqHJ} it is convenient to define the nonlinear operator
$$
T_c(H,a):=2(1+c)^{1/2}L_\om H+|\nabla H|^2+2\nabla H\cdot (a\nabla\vp)+f-b|Y|^2-c(1-|Y|^2)\,.
$$
The definition of the constant $b$ ensures that $[T_c(H,a)]_{\TT^2}=0$.

To study the operator $T_c$ we will make use of the Banach space
$\dot\cH_\rho$ of holomorphic functions~$H$ on the complex strip
\begin{equation}\label{Derho}
  \Delta(\rho):=\{\vp : \text{Re}\,\vp\in\TT^2\,,\; |\text{Im}\,\vp|<\rho\}
\end{equation}
that have zero mean on $\{\text{Im}\,\vp=0\}$ (i.e.,
$[H]_{\TT^2}=0$). This space is endowed with the supremum norm
$\|H\|_\rho := \sup_{\vp\in \Delta(\rho)} |H(\vp)|$. We will also
denote by $\dot\cH_\rho$ an analogous space of vector or matrix-valued
functions. With some abuse of notation, we still use the notation
$\vp$ for the complexification of the toroidal coordinates~$\vp$, and the same name for a function (vector or matrix) on $\TT^2$ and its complexification on $\Delta(\rho)$. Since the induced metric on $\pd\Om$ is analytic and $Y$ is also an analytic vector field, it then follows that there is some $\rho_0>0$ such that $T_c$ defines a map
\[
T_c: \dot\cH_\rho \to \dot\cH_\rho
\]
for all $\rho<\rho_0$.

To solve the equation
\begin{equation}\label{eqTcH}
T_c(H,a)=0\,,
\end{equation}
we will crucially use the additional requirement that $X$ is
analytically conjugate to the linear field $(1+c)^{1/2}\om$. More
precisely, let us consider the equation
\begin{equation}\label{eqreduc}
\Phi^*\Big((1+c)^{1/2}\om+\nabla H+a\nabla\vp\Big)-(1+c)^{1/2}\om=0\,,
\end{equation}
where $\Phi(\vp):=\vp+v(\vp)$ is a diffeomorphism of $\TT^2$ and
$v:\TT^2\to \RR^2$. We denote the LHS of this equation by $R_c(H,v,a)$.

Our goal is to find analytic solutions  $(H,v,a)$ to Equations~\eqref{eqTcH}
and~\eqref{eqreduc} when $|c|$ is small enough. Notice that~\eqref{eqreduc} automatically guarantees that condition~(ii) is satisfied.

To solve this equation, we apply Lemma~\ref{L:Newton} to the
approximate solution
$$(H_0,v_0,a_0):=(0,0,0)\,.$$
We shall now use the
notation of this lemma without further notice. As
$T_c(H_0,a_0)=-c(1-|Y|^2)$ and $R_c(H_0,v_0,a_0)=0$, it is clear that
\[
  E_0=\|1-|Y|^2\|_{\rho}|c|\,.
\]
It is obvious that there is $c_0>0$ small enough such that the
assumption~\eqref{eqE} holds for all $|c|<c_0$ (of course, the smallness assumption on $v_0$ is also satisfied because $\|v_0\|_{\rho}=0$). It remains to check the generic condition on the matrix
$M^{(0)}$, cf. Equation~\eqref{eq.genM}. Since $\Phi_0=\id$, an easy
computation shows that the columns of the $2\times2$ matrix $M^{(0)}$ are given
by the vectors
\[
[\nabla \vp_i-\om_i\nabla L_\om^{-1}(1-|Y|^2)]_{\TT^2}\,,
\]
with $i=1,2$. In terms of the positive definite symmetric matrix $G$ describing
the metric on~$\pd\Om$ in the $\vp$-coordinates, it is immediate to check that Equation~\eqref{eq.genM} is equivalent to
\[
\det\Bigg[G(\vp)^{-1}\cdot\left(
    \begin{array}{cc}
      1-\om_1\partial_{\vp_1}\mathcal R & -\om_2\partial_{\vp_1}\mathcal R \\
      -\om_1\partial_{\vp_2}\mathcal R & 1-\om_2\partial_{\vp_2}\mathcal R \\
    \end{array}
  \right)\Bigg]_{\TT^2}\neq 0
\]
where $\mathcal R(\vp)$ is the unique zero mean solution to the equation
\[
\om_1\pd_{\vp_1}\mathcal R+\om_2\pd_{\vp_2}\mathcal R=1-|Y|^2\,.
\]
This condition is immediately satisfied, by Definition~\ref{D:torus}, if $B$ is a
nondegenerate Beltrami field of type I on the toroidal domain $\Om$.

For any $\rho'\in(0,\rho)$, we can then conclude from
Lemma~\ref{L:Newton} that there exists a unique
triple  $(H,v,a)\in \dot\cH_{\rho'}
\times\dot\cH_{\rho'}\times\RR^2$ in a neighborhood of~$(0,0,0)$ such that Equations~\eqref{eqTcH}
and~\eqref{eqreduc} hold, provided that $|c|$ is small enough. It is
clear that $(H,v,a)$ depends continuously on~$c$.

The bound~(iii) then follows from the estimate~\eqref{eqbound} below, the usual Cauchy estimate
\[
\|H\|_{C^r(\TT^2)}+|a|\leq C_r \|H\|_{\rho'}+|a|\leq C|c|\,,
\]
and the obvious bound
\[
|b|\leq C(|a|+|a|^2+|a|\|H\|_{C^1(\TT^2)}+\|H\|_{C^1(\TT^2)}^2)\leq C|c|
\]
for~$|c|<1$.
\end{proof}

\subsection{Existence of solutions of the Hamilton-Jacobi equation}\label{SS.prop}

In this section we prove the basic lemma used to establish the
existence of analytic solutions to
Equations~\eqref{eqTcH}-\eqref{eqreduc}. To this end, note that
Equation~\eqref{eqreduc} reads as
\begin{equation}\label{eqred2}
R_c(H,v,a):=(1+c)^{1/2}L_\om v-\nabla H\circ(\id+v)-(a\cdot\nabla\vp)\circ(\id+v)=0\,.
\end{equation}
Here and in what follows, when the operator $L_\om$ acts on vector or
matrix-valued functions, its action is understood componentwise.

To solve the system
\begin{equation}\label{TR}
T_c(H,a)=0\,, \qquad R_c(H,v,a)=0\,,
\end{equation}
we shall use Newton's quadratic scheme and the R{\"u}ssmann estimates
for analytic solutions to cohomological equations~\cite{Rus}. We recall that the
constants $\ga>0$ and $\tau>1$ in the proof appear in the
definition of the Diophantine vector $\om$, and that one can assume $\ga\leq1$ without any loss of generality.

\begin{lemma}\label{L:Newton}
  Let us take $c\in[-\frac12,\frac12]$ and consider a triple
  $(H_0,v_0,a_0)\in(\dot\cH_\rho, \dot\cH_\rho,\RR^2)$. For any
  $\rho'\in(0,\rho)$, if
  $\|v_0\|_{\rho}$ and
  \begin{equation}\label{eqE}
E_0:=\|T_c(H_0,a_0)\|_{\rho}+\|R_c(H_0,v_0,a_0)\|_{\rho}
\end{equation}
are smaller than a certain constant $\ep_0>0$ that depends on~$\rho'$
but not on~$c$, and if the approximate solution $(H_0,a_0,v_0)$
satisfies the generic assumption given by Equation~\eqref{eq.genM}
below, then there exists a unique solution
$(H,v,a)\in \dot\cH_{\rho'} \times\dot\cH_{\rho'}\times \RR^2$ to
Equations~\eqref{eqTcH} and~\eqref{eqred2} (or, equivalently,
\eqref{TR}) bounded as
\begin{equation}\label{eqbound}
\|H-H_0\|_{\rho'}+\|v-v_0\|_{\rho'}+|a-a_0|<CE_0\,.
\end{equation}
\end{lemma}

\begin{proof}
To set a quadratic Newton iteration, we introduce corrections
$(\xi,\eta,\al)$ to the approximate solution so that
\[
(H_1,v_1,a_1):=(H_0,v_0,a_0)+ (\xi,\eta,\al)
\]
is a solution to the equations modulo a quadratic error, which is
bounded by $CE_0^2$ (precise estimates will be shown later). We also take a constant $b_0$ ensuring that $[T_c(H_0,a_0)]_{\TT^2}=0$, and introduce the correction
\[
b_1:=b_0+\beta\,,
\]
where the constant $\beta$ will be fixed later.

Setting $E_H^0:=T_c(H_0,a_0)$ and $E_v^0:=R_c(H_0,v_0,a_0)$, we then obtain $(\xi,\eta,\al)$ as solutions to the linearized equations
\begin{multline}\label{eqlin1}
2X_0\cdot\nabla\xi +2\nabla H_0\cdot (\alpha\nabla\vp)+2(a_0\nabla\vp)\cdot(\alpha\nabla\vp)\\
+2(1+c)^{1/2}(\alpha\om)-\be|Y|^2=-E_H^0
\end{multline}
and
\begin{multline}\label{eqlin2}
(1+c)^{1/2}L_\om\eta-\Big(D(\nabla H_0)\circ(\id+v_0)\Big)\eta-(\al\nabla\vp)\circ(\id+v_0)\\
-\Big(D(a_0\nabla\vp)\circ(\id+v_0)\Big)\eta=-E_v^0+\nabla\xi\circ(\id+v_0)\,.
\end{multline}
In Equation~\eqref{eqlin1}, the vector field $X_0$ is defined as
\[
X_0:=(1+c)^{1/2}\om+\nabla H_0+a_0\nabla\vp\,,
\]
and the constant $\beta$ will be chosen later to ensure the
solvability of the equation (that is, so that a certain zero mean
condition holds). In Equation~\eqref{eqlin2}, the symbol $D$ is used
to denote the Jacobian matrix of a vector field.

Taking the pullback of Equation~\eqref{eqlin1} with the diffeomorphism $\Phi_0:=\id+v_0$, defining the function $\hat\xi:=\xi\circ\Phi_0$, and using that
\[
\Phi^*_0X_0=(1+c)^{1/2}\om-(I+Dv_0)^{-1}E_v^0\,,
\]
we can rewrite~\eqref{eqlin1} as
\begin{multline*}
2(1+c)^{1/2}L_\om\hat\xi-2\Big((I+Dv_0)^{-1}E_v^0\Big)\hat\xi\\
=\Phi_0^*\Big(-E_H^0+\be|Y|^2-2\nabla H_0\cdot (\alpha\nabla\vp)-2(a_0\nabla\vp)\cdot(\alpha\nabla\vp)-2(1+c)^{1/2}(\alpha\om)\Big)\,.
\end{multline*}
In this equation, $I$ denotes the $2\times 2$ identity matrix. We also
observe that if $\|v_0\|_\rho$ is small enough, the matrix $I+Dv_0$ is
invertible. The second summand, which denotes the action of the vector
field $-2(I+Dv_0)^{-1}E_v^0$ (understood as a first order differential
operator) on the function $\hat\xi$, is in fact a quadratic term;
precise estimates will be given below. Thus, we can drop this term and consider the following equation:
\begin{multline}\label{eqlin1b}
2(1+c)^{1/2}L_\om\hat\xi\\
=\Phi_0^*\Big(-E_H^0+\be|Y|^2-2\nabla H_0\cdot (\alpha\nabla\vp)-2(a_0\nabla\vp)\cdot(\alpha\nabla\vp)-2(1+c)^{1/2}(\alpha\om)\Big)\,.
\end{multline}

Following Zehnder~\cite{Zehnder}, to study Equation~\eqref{eqlin2} we define a new function $\widetilde\eta$ as
\[
\eta=:(I+Dv_0)\widetilde\eta\,.
\]
Computing the Jacobian matrix of the equation that defines $E_v^0$,
one obtains the identity
\begin{multline*}
DE_v^0=(1+c)^{1/2}L_\om(Dv_0)-\Big(D(\nabla H_0)\circ(\id+v_0)\Big)(I+Dv_0)\\
-\Big(D(a_0\nabla\vp_0)\circ(\id+v_0)\Big)(I+Dv_0)\,.
\end{multline*}
Plugging this expression into Equation~\eqref{eqlin2}, and dropping the term $(DE_v^0)\widetilde\eta$, which is quadratic, we can write
\begin{align}\label{eqlin2b}
L_\om\widetilde\eta=\frac{1}{(1+c)^{1/2}}(I+Dv)^{-1}\Big((\al\nabla\vp+\nabla\xi)\circ(\id+v_0)-E_v^0\Big)\,.
\end{align}

Summarizing, we have changed the original linearized system of equations by the equivalent linear cohomological Equations~\eqref{eqlin1b} and~\eqref{eqlin2b}. Choosing the constant $\beta$ in Equation~\eqref{eqlin1b} so that
\[
\Big[\Phi_0^*\Big(-E_H^0+\be|Y|^2-2\nabla H_0\cdot (\alpha\nabla\vp)-2(a_0\nabla\vp)\cdot(\alpha\nabla\vp)-2(1+c)^{1/2}(\alpha\om)\Big)\Big]_{\TT^2}=0\,,
\]
Equation~\eqref{eqlin1b} admits a unique zero-mean solution $\xi$,
depending on the constant vector $\alpha$, which is of the form
\begin{equation*}
\xi=\xi^E+(\xi_1^H+\xi_1^a+\xi_1^\om)\alpha_1+(\xi_2^H+\xi_2^a+\xi_2^\om)\alpha_2\,.
\end{equation*}
Here
\begin{align*}
\xi^E&:=\Phi_{0*}L_\om^{-1}\Phi_0^*\Bigg(\frac{\beta_{0}|Y|^2-E_H^0}{2(1+c)^{1/2}}\Bigg)\,,\\
\xi^H_i&:=\Phi_{0*}L_\om^{-1}\Phi_0^*\Bigg(\frac{\beta_{1}^{(i)}|Y|^2-\nabla H_0\cdot\nabla\vp_i}{(1+c)^{1/2}}\Bigg)\,,\\
\xi^a_i&:=\Phi_{0*}L_\om^{-1}\Phi_0^*\Bigg(\frac{\beta_{2}^{(i)}|Y|^2-(a_0\nabla\vp)\cdot\nabla\vp_i}{(1+c)^{1/2}}\Bigg)\,,\\
\xi^\om_i&:=\Phi_{0*}L_\om^{-1}\Phi_0^*\Bigg(\beta_{3}^{(i)}|Y|^2-\om_i\Bigg)\,,
\end{align*}
and the constants
$\beta_{0},\beta_{1}^{(i)},\beta_{2}^{(i)},\beta_{3}^{(i)}$ (with
$i=1,2$) guarantee that all the above functions of the form
$\Phi_0^*\big(\cdots\big)$ have zero mean. This ensures that the
action of the operator $L_\om^{-1}$ (mapping functions of zero mean to
functions of zero mean) is well defined. Note, in particular,
\[
\beta=\beta_{0}+(\beta_{1}^{(1)}+\beta_{2}^{(1)}+\beta_{3}^{(1)})\alpha_1+(\beta_{1}^{(2)}+\beta_{2}^{(2)}+\beta_{3}^{(2)})\alpha_2\,.
\]

Next, let us plug the expression for $\xi$ in Equation~\eqref{eqlin2b} and
consider the $2\times 2$ matrix-valued function $M^{(0)}$ whose
columns are the vector fields
\[
\nabla\vp_i+\nabla(\xi^H_i+\xi^a_i+\xi^\om_i)
\]
with $i=1,2$. The solvability of Equation~\eqref{eqlin2b} then follows
if and only if one can pick a vector $\alpha\in\RR^2$ such that
\begin{equation}\label{eqalpha}
\Big[(I+Dv_0)^{-1}M^{(0)}\circ(\id+v_0)\Big]_{\TT^2}\alpha=\Big[(I+Dv_0)^{-1}\big(E_v^0-\nabla \xi^E\circ(\id+v_0)\big)\Big]_{\TT^2}\,.
\end{equation}
This linear equation has a solution if and only if the matrix $M^{(0)}$ satisfies the invertibility condition
\begin{equation}\label{eq.genM}
\det\Big[(I+Dv_0)^{-1}M^{(0)}\circ(\id+v_0)\Big]_{\TT^2}\neq0\,.
\end{equation}
This is the generic assumption appearing in the statement of the
lemma. Note that Equation~\eqref{eq.genM} only depends on
$(H_0,v_0,a_0)$, on the vector field $Y$ and on the domain $\Om$.

We have then proved that, fixing the constants $\alpha$ and $\beta$
as above, there is a unique solution $(\xi,\widetilde\eta)$ to the
linearized equations~\eqref{eqlin1b}-\eqref{eqlin2b}. Now let us
estimate the analytic norms of these solutions to show that this
scheme is indeed quadratic, and that it can be iterated because the
norms of the corrected approximate solutions are uniformly
bounded. For this we use R\"ussmann estimates~\cite{Rus}:
if $A\in \dot\cH_\rho$ and $\om$ is a Diophantine vector, for
each~$\de>0$ the cohomological equation $L_\om B =A$
has a unique solution $B\in \dot\cH_{\rho-\delta}$ that can be bounded as
$$\|B\|_{\rho-\delta} \leq C \gamma^{-1} \delta^{-\tau} \|A\|_\rho\,.$$
The constant $C$ is independent of~$\de$

In what follows, let us fix a small constant $0<\delta<1$, and $\delta<\rho$ (which
will measure the loss of analytic band in the iteration) and assume
that $E_0<\ep_0$ (cf. Equation~\eqref{eqE}), where $\ep_0$ satisfies the
smallness condition
\begin{equation}\label{eqE2}
{\ep_0}\ll \gamma^6\delta^{9+6\tau}\,.
\end{equation}
By assumption, $\|v_0\|_\rho<\ep_0$, and hence~\eqref{eqE2} implies that $\|v_0\|_\rho\ll\de$, which guarantees, by the Cauchy estimate, that $I+Dv_0$ is close to the unit matrix.

A straightforward computation using R\"ussmann estimates and the Cauchy estimate for derivatives then implies
\begin{align*}
\|\xi^E\|_{\rho-\de}&\leq C\ga^{-1}\de^{-\tau}\|E_H^0\|_{\rho}\,,\\
\|\xi_i^H\|_{\rho-2\de}&\leq C\ga^{-1}\de^{-\tau-1}\|H_0\|_\rho\,,\\
\|\xi_i^a\|_{\rho-\de}&\leq C\ga^{-1}\de^{-\tau}|a_0|\,,\\
\|\xi_i^\om\|_{\rho-\de}&\leq C\ga^{-1}\de^{-\tau}\,.
\end{align*}
Here we have used that the condition $\|v_0\|_\rho\ll\de$ ensures that
the diffeomorphism $\Phi_0$ is close to the identity, and hence the
constant $C$ can be taken independent of $v_0$. The constant does
depend on $Y$ and $\Om$, though.

Solving $\alpha$ in Equation~\eqref{eqalpha}, one then concludes
\begin{align*}
&|\alpha|\leq C\Big(\|E_v^0\|_{\rho}+\ga^{-1}\de^{-1-\tau}\|E_H^0\|_{\rho}\Big)\leq C\ga^{-1}\de^{-1-\tau}E_0\,,\\
&\|\xi\|_{\rho-2\de}\leq C\Big(\ga^{-1}\de^{-\tau}\|E_H^0\|_{\rho}+\ga^{-2}\de^{-2-2\tau}E_0\Big)\leq C\ga^{-2}\de^{-2-2\tau}E_0\,.
\end{align*}
Now, solving $\widetilde\eta$ in Equation~\eqref{eqlin2b}, and using
again that $\|v_0\|_{\rho}\ll\de$, we readily estimate~$\eta$ as
\begin{equation*}
\|\eta\|_{\rho-4\de}\leq C\ga^{-3}\de^{-3-3\tau}E_0\,.
\end{equation*}
Analogously, the constant $\beta$ in Equation~\eqref{eqlin1b} can be bounded as
\[
|\beta|\leq C\Big(\|E_H^0\|_{\rho}+\de^{-1}|\al|\Big)\leq C\ga^{-1}\de^{-2-2\tau}E_0\,.
\]
In these bounds, the constant $C$ only depends on $\|H_0\|_\rho$, $|a_0|$, $Y$ and $\Om$.

In order to show that this scheme can be iterated, let us check that
the norms of the corrected approximate solution $(H_1,v_1,a_1)$ and the constant $b_1$ remain uniformly bounded. Indeed, if $\|H_0\|_\rho+|a_0|+|b_0|<\varkappa_0$ for some positive constant $\varkappa_0$, we can easily derive that
\begin{align*}
&\|H_1\|_{\rho-2\de}+|a_1|+|b_1|\leq \|H_0\|_{\rho}+|a_0|+|b_0|+C\ga^{-2}\de^{-2-2\tau}E_0<\varkappa_0\,,
\end{align*}
where in the last inequality we have used the assumptions~\eqref{eqE}
and~\eqref{eqE2}. Analogously, if $E_0$ is small enough (cf. Equation~\eqref{eqE2}),
\begin{align*}
\|v_1\|_{\rho-4\de} &\leq \|v_0\|_{\rho}+C\ga^{-3}\de^{-4-3\tau}E_0\ll \de\,,\\
\|(I+Dv_1)-(I+Dv_0)\|_{\rho-5\de}&=\|D\eta\|_{\rho-5\de}\leq C\ga^{-3}\de^{-4-3\tau}E_0\ll 1\,.
\end{align*}
The generic assumption~\eqref{eq.genM} is then satisfied in the
iteration because the difference $M^{(1)}-M^{(0)}$ is bounded as
\[
\|M^{(1)}-M^{(0)}\|_{\rho-5\de}\leq C\ga^{-3}\de^{-4-3\tau}E_0\ll 1\,.
\]
Therefore, if $\de$ is small enough,
\begin{multline*}
\Big|\det\Big[(I+Dv_1)^{-1}M^{(1)}\circ(\id+v_1)\Big]_{\TT^2}\Big|\\
\geq \Big|\det\Big[(I+Dv_0)^{-1}M^{(0)}\circ(\id+v_0)\Big]_{\TT^2}\Big|-C\de>0
\end{multline*}
because $M^{(0)}$ satisfies an invertibility condition by hypothesis.

To complete the proof of the proposition, we have to check that the new
errors $E_H^1:=T_c(H_1,a_1)$ and $E_v^1:=R_c(H_1,v_1,a_1)$ are
quadratic with respect to the errors $E_H^0$ and $E_v^0$. This follows
from the fact that
\begin{align*}
E_H^1&=(\nabla\xi)^2+2\nabla\xi\cdot(\alpha\nabla\vp)+(\alpha\nabla\vp)^2-2\Big((I+Dv_0)^{-1}E_v^0\Big)\xi\circ\Phi_0\,,\\
E_v^1&=\Big(D^2(\nabla H_0+a_0\nabla\vp)\circ\Phi_0\Big) \eta\otimes\eta\\
&\qquad\qquad+\Big(D(\nabla\xi-\al\nabla\vp)\circ\Phi_0\Big)\eta+\Big(DE_v^0(I+Dv_0)^{-1}\Big)\eta\,.
\end{align*}
A straightforward computation using the previous estimates
shows
\begin{align*}
\|E_H^1\|_{\rho-3\de}&\leq C\ga^{-4}\de^{-6-4\tau}E_0^2\\
 \|E_v^1\|_{\rho-4\de}&\leq C\ga^{-6}\de^{-9-6\tau}E_0^2\,,
\end{align*}
so the scheme is indeed quadratic. In particular,
\[
E_1:=\|E_H^1\|_{\rho-3\de}+|E_v^1\|_{\rho-4\de}\leq C\ga^{-6}\de^{-9-6\tau}E_0^2\,.
\]
so the new error $E_1$ is smaller than~$\ep_0$ because of the smallness condition~\eqref{eqE2}.

It is now standard that the scheme can therefore be iterated to yield a unique solution
$(H,v,a)\in \dot\cH_{\rho'}\times \dot\cH_{\rho'} \times \RR^2$ to
Equations~\eqref{eqTcH} and~\eqref{eqred2} that is bounded as in
Equation~\eqref{eqbound}, with $\rho':=\rho-8\de$. This is an easy
consequence of the previous estimates and the following well known lemma:

\begin{lemma}\label{lem:conv2}
Let $\{E_n\}_{n=0}^\infty$ be a sequence of positive real numbers such that
\[
E_{n+1}\leq C\ga^{-a}\de_n^{-b-a\tau}E_n^2
\]
for some constant $C>0$, positive reals $a,b$, and small constant
$\de_0$, with $\de_{n+1}:={\de_n}/{2}$. Then $E_n\to0$ as
$n\to\infty$ provided that
${E_0}\ll {\gamma^a \de^{b+a\tau}} $.
\end{lemma}
The result is then proven.
\end{proof}

\section{A KAM theorem adapted to divergence-free vector fields}\label{sec:teo}

In this section we prove a KAM theorem that applies to divergence-free
vector fields in $\RR^3$ that are not necessarily close to integrable
ones or in Birkhoff normal form. Another key technical advantage of
this result is that the twist condition is written in terms of the
(approximately) invariant torus and the vector field itself, so one
does not need any fine information about the trajectories of the
vector field. The proof follows the parametrization method as presented in~\cite{LGJV} in the context of Hamiltonian systems.

It is convenient to study the reducibility properties of the invariant
tori using embeddings of $\TT^2\to\RR^3$, denoting by~$\vp$ the
natural coordinates on the torus. As before, in
this section $D$ denotes the Jacobian matrix of a vector-valued
function and $|\cdot|$ stands for its norm computed with the
induced metric.

\subsection{Statement of results}

Let $B$ be an analytic divergence-free vector field in $\RR^3$ and let
$\omega \in \RR^2$ be a frequency vector satisfying the Diophantine
condition~\eqref{eq:diof}. Let $K:\TT^2 \to\RR^3$ be an analytic
embedding, which parametrizes a torus $\cT:=K(\TT^2)\subset\RR^3$
that we will eventually assume to be approximately invariant under $B$
with frequency $\omega$. The associated error, or defect, of invariance is measured using the function
\begin{equation}\label{eq:error}
E(\vp):=L_\omega K(\vp) -B(K(\vp))\,,
\end{equation}
where, as before, $L_\omega K(\vp):= DK(\vp) \omega$ is the derivative of~$K$ in
the direction of~$\om$.
We still denote by
\begin{equation}
G_K(\vp):=DK(\vp)^\top DK(\vp)
\end{equation}
the matrix representation of the pull-back to $\TT^2$ of the Euclidean
metric. It is obviously a nonsingular $2\times2$ matrix because $K$ is
an embedding. As $K$ and~$B$ are analytic, they can be analytically
continued to a complex strip that we call $\Delta(\rho)$ of the form~\eqref{Derho}.

An important ingredient of a KAM theorem is its twist condition, which
in this case we define in terms of the embedding $K$ and the vector
field $B$ as follows:

\begin{definition}\label{def:ndeg}
The embedding $K:\TT^2\to\RR^3$ is {\em twist}\/ for the vector field
$B$ if the twist constant
\begin{equation*}
T:=[A_1]_{\TT^2}\om_2-[A_2]_{\TT^2}\om_1\equiv[A]_{\TT^2}\cdot \omega^\perp
\end{equation*}
is nonzero. Here the $2\times 1$ matrix $A$ is defined as
\begin{equation}\label{eq:condA}
A(\vp) :=
 - \frac{G_K(\vp)^{-1}
DK(\vp)^\top [DB(K(\vp))^\top + DB(K(\vp))]\, n(\vp)}{|n(\vp)|^2}
\end{equation}
and $n$ is the vector field normal to the torus $\cT$ given by
\begin{equation}\label{eq:n}
n(\vp):=\partial_{\vp_1}K(\vp) \times \partial_{\vp_2}K(\vp)\,.
\end{equation}
\end{definition}

The main result of this section, which is of interest in itself, is
the following KAM theorem. We recall that two tori are said to be {\em
  $\eta$-close}\/ if
there is a diffeomorphism $\Psi$ of~$\RR^3$ mapping one into the other
and such that $\|\Psi-\id\|_{C^r}<\eta$, where $r\geq4$ is any fixed
integer.

\begin{theorem}\label{teo:kam:div}
Let $B$, $\omega$, $\rho$ and $K$ be as above. Assume that the embedding $K$ is twist with respect to the vector field $B$.
%Then if the initial invariance error satisfies
%\begin{equation}\label{eq:small}
%\frac{\|E\|_\rho}{\gamma^4 \de^{4\tau}} \ll 1\,,
%\end{equation}
If the invariance error $\|E\|_\rho$ is small enough, there is a
constant $\lambda_*$ and an analytic Diophantine invariant torus $\cT_*$ of~$B$
with frequency vector $\omega_* :=
(1+\lambda_*)\,\omega$. Furthermore, $|\la_*|\leq C \|E\|_\rho$ and
the torus $\cT_*$ is $(C \|E\|_\rho)$-close to~$\cT:=K(\TT^2)$.
\end{theorem}

%The correction $\delta_\omega$ plays the role of a parameter that allows us
%to solve the cohomological equations involved in approximating the inverse of $\cR$.
%This correction is necessary because we do not have enough internal parameters to control the two frequencies
%of the torus (there is a single ``action'' labeling invariant tori in divergence-free vector fields in $\RR^3$). The good point is that, as we shall discuss in Section~\ref{sec:app:inv}, it is possible to select $\delta_\omega$
%to be proportional to $\omega$. A consequence of this is that the iterative procedure preserves
%the Diophantine properties of the frequency, so we do not need to exclude a subset of parameters,
%which is a technically more demanding case.

The following corollary, which is a standard consequence of
Theorem~\ref{teo:kam:div}, is the result we employed in the proof of Theorem~\ref{T:main}.

\begin{corollary}\label{Cor_KAM}
Assume that the (analytic, divergence-free) vector field $B$ has an
invariant torus $\cT$ with Diophantine frequency vector
$\om$. If $\cT$ is twist (in the sense of Definition~\ref{def:ndeg}),
then it is accumulated (in both connected components of the complement
$\RR^3\backslash\cT$) by a set of Diophantine analytic
invariant tori of~$B$.
\end{corollary}
\begin{proof}
Simply apply Theorem~\ref{teo:kam:div} to the triple $(B,\om',K)$,
with a Diophantine frequency vector $\om'\neq \om$ which is very close
to $\om$ and has the same Diophantine constants $(\ga,\tau)$. This
ensures the invariance error is as small as one wishes. The accumulation property of the statement follows from the fact that the set of Diophantine numbers with fixed $(\ga,\tau)$ has positive measure in any neighborhood of~$\om$. The invariant torus $\cT'$ with frequency vector $(1+\la')\om'$ that one obtains with Theorem~\ref{teo:kam:div}, lies in the exterior component of $\RR^3\backslash\cT$ or in the interior one depending on whether $\frac{\om'_2}{\om'_1}$ is smaller or bigger than $\frac{\om_2}{\om_1}$ (because of the twist condition).
\end{proof}

To prove Theorem~\ref{teo:kam:div} we will iteratively correct
the embedding and the frequency vector by means of the Newton
method. Denoting the corrected quantities by
\begin{align*}
  \bar K(\vp)&:=K(\vp)+\Delta_K(\vp)\,,\\
  \bar \omega&:=\omega+\delta_\omega\,,
\end{align*}
one is led to choosing $\Delta_K(\vp)$ as a solution of the linearized equation
\begin{equation}\label{eq:cR}
\cR(\Delta_K(\vp)) := L_\omega \Delta_K(\vp) - DB(K(\vp)) \Delta_K(\vp) = - E(\vp)-DK(\vp) \delta_\omega.
\end{equation}

Our next goal is to analyze this equation, which will involve
developing a geometric setting in which the analytic properties of the equation
are laid bare. The most efficient way to do this is by first
considering the (trivial) case where $\cT$ is an actual invariant
torus of~$B$. This case will be considered in the next
subsection. Subsequently, in Subsection~\ref{sec:app:inv} we will
refine this approach to deal with the approximately invariant case,
which will enable us to prove Theorem~\ref{teo:kam:div}.

\subsection{Geometric study of the invariant case}\label{sec:inv}

In this subsection we shall assume that $\calT$ is an invariant torus
of the (divergence-free) vector field $B$ with frequency vector
$\omega \in \RR^2$, which we parametrize by the map $K: \TT^2 \to \RR^3$. The invariance equation reads as
\begin{equation}\label{eq:inv1}
L_\omega K(\vp) = B(K(\vp))\,.
\end{equation}

The columns of the matrix $DK$ are obviously a basis of the tangent
space $T_{K(\vp)} \calT$. With $n(\vp)$ being the normal
vector~\eqref{eq:n}, the key geometric observation is that the frame
\begin{equation}\label{eq:basis1}
\left(DK(\vp), \frac{n(\vp)}{|n(\vp)|^2} \right)
\end{equation}
greatly simplifies the analysis of the operator~$\cR$.
More precisely, in this frame~\eqref{eq:basis1}, the linear
operator~$\cR$ has a triangular structure that reduces the study of
Equation~\eqref{eq:cR} to that of two cohomological equations
with constant coefficients.

To prove this, we start by computing the action of the operator $\cR$ on the frame~\eqref{eq:basis1}.
Firstly, by taking derivatives of both sides of Equation~\eqref{eq:inv1}, we obtain that
\[
L_\omega DK(\vp) = DB(K(\vp)) DK(\vp)\,,
\]
which implies that $\cR(DK)=0$. To compute $\cR (\frac{n(\vp)}{|n(\vp)|^2})$, we use the following lemma:

\begin{lemma}\label{lem:trace}
If Equation~\eqref{eq:inv1} is satisfied,
\[
L_\omega n(\vp)=-DB(K(\vp))^\top n(\vp).
\]
\end{lemma}

\begin{proof}
A direct computation yields
\begin{align*}
L_\omega n(\vp) = {} & L_\omega (\partial_{\vp_1} K(\vp)) \times \partial_{\vp_2} K(\vp)
             + \partial_{\vp_1} K(\vp) \times L_\omega (\partial_{\vp_2} K(\vp)) \\
           = {} &
             (DB(K(\vp)) \partial_{\vp_1} K(\vp) \times \partial_{\vp_2} K(\vp)
             + \partial_{\vp_1} K(\vp) \times (DB(K(\vp)) \partial_{\vp_2} K(\vp)\,.
\end{align*}
Then the elementary identity
\[
(U v_1) \times v_2 + v_1 \times (U v_2) + U^\top (v_1 \times v_2) = 0\,,
\]
holds if $U$ is a $3\times 3$ traceless matrix and $v_1$, $v_2$ are
$3\times 1$ vectors. If we take $U$ as $DB(K(\vp))$, and note that
$DB$ has zero trace because $B$ is divergence-free, the lemma then follows.
\end{proof}

We are ready to compute the action of the linear operator $\cR$ on the
normal vector:
\begin{align}
\cR\left( \frac{n(\vp)}{|n(\vp)|^2} \right)
&=  L_\omega \left( \frac{n(\vp)}{|n(\vp)|^2} \right)
       - \frac{DB(K(\vp)) n(\vp)}{|n(\vp)|^2}
	 \nonumber \\
&= \frac{L_\omega n(\vp)}{|n(\vp)|^2}
       + n(\vp) L_\omega [ (|n(\vp)|^2)^{-1} ]
       - \frac{DB(K(\vp)) n(\vp)}{|n(\vp)|^2}
	 \nonumber\\
&= \frac{L_\omega n(\vp)-DB(K(\vp))n(\vp)}{|n(\vp)|^2}
       - \frac{n(\vp) (L_\omega n(\vp)^\top
	 n(\vp)+n(\vp)^\top L_\omega n(\vp))}{|n(\vp)|^4}\nonumber\\
& = -\frac{[DB(K(\vp))^\top+DB(K(\vp))]\,n(\vp)}{|n(\vp)|^2}\nonumber\\
&\qquad\qquad+\frac{n(\vp)^\top[DB(K(\vp))^\top+DB(K(\vp))]\,n(\vp)}{|n(\vp)|^4}\,n(\vp)\,. \label{eq:Rn_exp}
\end{align}
Here we have used Lemma~\ref{lem:trace} to pass to the last line. It is straightforward to check that this
expression can be written in the frame~\eqref{eq:basis1} as
\begin{equation}\label{eq:Rn1}
\cR\left(\frac{n(\vp)}{|n(\vp)|^2} \right) = DK(\vp) A(\vp) \,,
\end{equation}
where the $2 \times 1$ matrix $A(\vp)$ is
\begin{equation}\label{eq:twist}
A(\vp) = - \frac{G_K(\vp)^{-1}
DK(\vp)^\top [DB(K(\vp))^\top + DB(K(\vp))]\, n(\vp)}{|n(\vp)|^2}\,.
\end{equation}
Observe $A(\vp)$ is therefore just as in Equation~\eqref{eq:condA}.

\subsection{Proof of Theorem~\ref{teo:kam:div}}\label{sec:app:inv}

Our goal now is to use the frame introduced in the previous section to
analyze Equation~\eqref{eq:cR}, which determines the small corrections
$\Delta_K$ and $\delta_\omega$. The problem ultimately boils down to studying the
inverse of the operator $\cR$.

Taking derivatives in Equation~\eqref{eq:error}, we obtain
\begin{equation}\label{eq:aprox}
L_\omega DK(\vp)=DB(K(\vp))DK(\vp)+DE(\vp)\,.
\end{equation}
Arguing as in the proof of Lemma~\ref{lem:trace} but using this
equation instead of~\eqref{eq:inv1}, we prove the following:

\begin{lemma}\label{L:aprox}
The normal vector $n(\vp)$ satisfies the equation
\[
L_\omega n(\vp)=-DB(K(\vp))^\top n(\vp)
  + \partial_{\vp_1} E(\vp) \times \partial_{\vp_2} K(\vp)
  + \partial_{\vp_1} K(\vp) \times \partial_{\vp_2} E(\vp)\,.
\]
\end{lemma}

This allows us to compute the quantity
$\cR (\frac{n(\vp)}{|n(\vp)|^2})$ as in~\eqref{eq:Rn_exp}:
\begin{equation}\label{eq:cR:normal}
\cR\left(\frac{n(\vp)}{|n(\vp)|^2} \right) = DK(\vp) (A(\vp)+A'(\vp)) + \frac{n(\vp)}{|n(\vp)|^2} b(\vp)\,,
\end{equation}
Here $A(\vp)$ is the vector~\eqref{eq:twist} and
\begin{align}\label{eq:B}
A'(\vp)&:=\frac{G_K(\vp)^{-1} DK(\vp)^\top(
\partial_{\vp_1} E(\vp) \times \partial_{\vp_2} K(\vp)
  + \partial_{\vp_1} K(\vp) \times \partial_{\vp_2}
        E(\vp))}{|n(\vp)|^2}\\
  \label{eq:b}
b(\vp) &:= -\frac{(\partial_{\vp_1} E(\vp) \times \partial_{\vp_2} K(\vp)
  + \partial_{\vp_1} K(\vp) \times \partial_{\vp_2}
  E(\vp))^\top n(\vp)}{|n(\vp)|^2}\,.
\end{align}

The following lemma shows how to get an
approximate solution of Equation~\eqref{eq:cR}, modulo a quadratic
error, using solutions to a pair of cohomological equations:

\begin{lemma}\label{lem:corr}
Suppose that the functions $\xi_1,\xi_2$ on~$\TT^2$ satisfy the cohomological equations
\begin{align}
L_\omega \xi_1(\vp) + A(\vp) \xi_2(\vp) = {} & -G_K(\vp)^{-1} DK(\vp)^\top E(\vp) -\delta_\omega\,, \label{eq:xi1}\\
L_\omega \xi_2(\vp) = {} & -n(\vp)^\top E(\vp)\,. \label{eq:xi2}
\end{align}
Then
\begin{equation}\label{eq:Delta}
\Delta_K(\vp) := DK(\vp) \xi_1(\vp)+\frac{n(\vp)}{|n(\vp)|^2} \,\xi_2(\vp)
\end{equation}
solves Equation~\eqref{eq:cR} modulo a quadratic error. More precisely,
\begin{equation}\label{eq:newerror}
\cR(\Delta_K(\vp))+E(\vp)+DK(\vp)\delta_\omega=DK(\vp) E_1(\vp)
+\frac{n(\vp)}{|n(\vp)|^2} E_2(\vp)
\end{equation}
with
\begin{align}
E_1(\vp) := {} & G_K(\vp)^{-1} DK(\vp)^\top DE(\vp) \xi_1(\vp) +A'(\vp) \xi_2(\vp)\,, \label{eq:E1}\\
E_2(\vp) := {} & n(\vp)^\top DE(\vp) \xi_1(\vp)+b(\vp) \xi_2(\vp)\,. \label{eq:E2}
\end{align}
\end{lemma}
\begin{remark}
We observe that, at least formally, the quantities $E_1(\vp)$ and $E_2(\vp)$ are quadratic in $E(\vp)$. This is obvious except, perhaps, for the presence of the constant $\delta_\omega$ in Equation~\eqref{eq:xi1}. The fact that this constant is of order $[E]_{\TT^2}$ is a consequence of the solvability condition that the average of Equation~\eqref{eq:xi1} must be zero (see the discussion after Equation~\eqref{eq:aver}).
\end{remark}

\begin{proof}
First we compute
\begin{align*}
\cR(\Delta_K(\vp)) &=  \cR(DK(\vp)) \xi_1(\vp) + DK(\vp) L_\omega \xi_1 + \cR\left(\frac{n(\vp)}{|n(\vp)|^2} \right) \xi_2(\vp) + \frac{n(\vp) L_\omega \xi_2}{|n(\vp)|^2}\\[1mm]
&=
DE(\vp) \xi_1 (\vp)
+ DK(\vp) \big(L_\omega \xi_1(\vp) + (A(\vp)+A'(\vp))\xi_2(\vp)\big)\\
&\qquad\qquad\qquad \qquad\qquad\qquad\qquad\qquad+ \frac{n(\vp)(L_\omega \xi_2(\vp)+b(\vp) \xi_2(\vp))}{|n(\vp)|^2}\,.
\end{align*}
Next, we plug this expression in Equation~\eqref{eq:newerror} and
read off the errors $E_1(\vp)$ and $E_2(\vp)$:
\begin{align*}
E_1(\vp) &= G_K(\vp)^{-1}DK(\vp)^\top DE(\vp) \xi_1(\vp) + L_\omega \xi_1(\vp)+(A(\vp)+A'(\vp)) \xi_2  \\
&\qquad\qquad\qquad\qquad\qquad+ G_K(\vp)^{-1} DK(\vp)^\top E(\vp) +
                                                                                                           \delta_\omega\,,\\
  E_2(\vp) &= n(\vp)^\top DE(\vp) \xi_1(\vp) + b(\vp) \xi_2(\vp) + L_\omega \xi_2(\vp) + n(\vp)^\top E(\vp)\,.
\end{align*}
Then expressions~\eqref{eq:E1} and~\eqref{eq:E2} follow from Equations~\eqref{eq:xi1} and~\eqref{eq:xi2}.
\end{proof}

To solve the cohomological equations~\eqref{eq:xi1}
and~\eqref{eq:xi2}, we need to ensure that the RHS of these equations
have zero mean. In the case of Equation~\eqref{eq:xi2}, a simple
computation shows that
\begin{align*}
[n^\top E]_{\TT^2} &= {}
\frac{1}{4\pi^2} \int_{\TT^2} n(\vp)^\top E(\vp) \,d\vp\\
&=
\frac{1}{4\pi^2} \int_{\TT^2} n(\vp)^\top [L_\omega K(\vp)-B(K(\vp))] \,d\vp \\
&= {}
-\frac{1}{4\pi^2} \int_{\TT^2} n(\vp)^\top B(K(\vp)) \,d\vp\,,
\end{align*}
where we have used that $n(\vp)^\top L_\omega K(\vp)=0$. The last
integral can be equivalently written as
\begin{align*}
\int_{\TT^2} n(\vp)^\top B(K(\vp)) \,d\vp&= \int_{\cT} N^\top B\, dS =\int_{\Omega} \Div B\, dx= 0\,,
\end{align*}
where $dS:=|n|\,d\vp$ is the induced area form on $\cT$,
$N:=\frac{n}{|n|}$ is the unit normal and $\Omega$ is the domain in
$\RR^3$ bounded by $\cT$. To obtain the last equality, we have also
used that $B$ is divergence-free. We thus conclude that $[n^\top
E]_{\TT^2}  =0$, so there exists a
unique zero mean solution~$\tilde\xi_2$ to Equation~\eqref{eq:xi2}.
The general solution to this equation is
\[
\xi_2(\vp)= \tilde \xi_2(\vp)+[\xi_2]_{\TT^2}
\]
with any constant $[\xi_2]_{\TT^2}\in\RR$.

Plugging the expression for~$\xi_2$ in Equation~\eqref{eq:xi1}, we
obtain that the average of the RHS of this equation is
\begin{equation}\label{eq:aver}
[A \tilde \xi_2]_{\TT^2}
+
[A]_{\TT^2} [\xi_2]_{\TT^2}
+
[G_K ^{-1} DK^\top E]_{\TT^2}
+
\delta_\omega\,.
\end{equation}
We now make the additional assumption
that $\de_\om$ and $\om$ are collinear, so
\[
  \delta_\omega:=\Lambda \omega
\]
for some $\La\in\RR$. This is crucial to preserve the Diophantine properties of the frequency vector. The twist condition in Definition~\ref{def:ndeg}
then ensures that there is a unique pair $([\xi_2]_{\TT^2},
\La)\in\RR^2$ for which the quantity~\eqref{eq:aver} is~0. It is immediate to check that the solution $([\xi_2]_{\TT^2},
\La)$ is of order $[E]_{\mathbb T^2}$.

We thus obtain an expression for the corrected embedding $\bar
K(\vp):=K(\vp)+\Delta_K(\vp)$. The invariance error associated with
the corrected embedding can be easily computed:
\begin{align*}
\bar E(\vp) := {} & L_{\bar\omega} \bar K(\vp)-B(\bar K(\vp))
=  \Big(\cR(\Delta_K(\vp))+E(\vp) + DK(\vp)\delta_\omega\Big) \\
& +\Big( B(K(\vp))+DB(K(\vp))\Delta_K(\vp)-B(K(\vp)+\Delta_K(\vp))\Big)+O(E^2)\\\
= {} & \Big(DK(\vp) E_1(\vp)
+\frac{n(\vp)}{|n(\vp)|^2} E_2(\vp)\Big) + O(E^2)\,.
\end{align*}
Therefore, the new error is quadratic in $E(\vp)$ by
Lemma~\ref{lem:corr}.

In view of the above estimates, it is now well known that, if a certain
smallness assumption is satisfied, the associated quadratic
scheme starting with $(K_0,\La_0):=(K,1)\in
\dot\cH_\rho\times\RR^+$ converges to some
$(K_*,\La_*)\in\dot\cH_{\rho'}\times\RR^+$. The embedded torus
$\cT_*:=K_*(\TT^2)$ is then a Diophantine invariant torus of~$B$ with
frequency vector $\om_*:=\La_*\om$. The details are standard (see e.g.~\cite{LGJV}), and go
essentially as in Section~\ref{SS.prop}, so we will just
sketch the argument.

R\"ussmann estimates for the solutions $\xi_1,\xi_2$ of the
cohomological
equations~\eqref{eq:xi1}-\eqref{eq:xi2} ensure that, for any small $\de>0$,
\begin{align*}
\|\xi_2\|_{\rho-\delta} &\leq C\gamma^{-1} \delta^{-\tau} \|E\|_\rho\,,\\
\|\xi_1\|_{\rho-2\delta}& \leq C\gamma^{-2} \delta^{-2\tau} \|E\|_\rho\,.
\end{align*}
The Cauchy estimate then permits to estimate $B,b,E_1,E_2$ and $\bar E$; in particular,
$$\|\bar E\|_{\rho-2\de} \leq C \gamma^{-4} \delta^{-4\tau} \|E\|^2_\rho\,.$$
These are the estimates for each step of the Newton method, which goes
as follows:
\begin{enumerate}
  \item Initialize the scheme with $(K_0,\La_0):=(K,1)$.
\item Given $(K_n,\La_n)$, set $\om_n:=\La_n\om$ and compute the
  invariance error $E_n:=L_{\omega_n} K_n-B\circ K_n$.
\item Construct the adapted frame by computing $DK_n$ and $n_n$.
\item Compute the new $2\times 1$ matrix $A_n$ as in Equation~\eqref{eq:twist}.
\item Solve the cohomological equations~\eqref{eq:xi1}
  and~\eqref{eq:xi2}, thus obtaining $\xi_{1,n}$ and
  $\xi_{2,n}$. The new constant vector $\Lambda_{n+1} \omega_n$ is then
  obtained from Equation~\eqref{eq:aver}.
\item Compute $\Delta_{K_n}$ from Equation~\eqref{eq:Delta}
  and set $K_{n+1}:=K_n + \Delta_{K_n}$.
\item Repeat the iteration step with
  $\de_{n+1}:=\de_n/2$.
\end{enumerate}
Lemma~\ref{lem:conv2} ensures the
existence of an invariant torus of~$B$ given by
\[
  (K_*,\om_*):=\lim_{n\to\infty}(K_n,\La_n\om)\in \dot\cH_{\rho-4\de}\times\RR^+
\]
provided that
${\|E_0\|_{\rho}}\ll{\ga^4\delta^{4\tau}}$,
so the theorem follows.

\section{Thin toroidal domains are generically
  nondegenerate}\label{S:nondeg}

As discussed in the Introduction, it is reasonable to expect that a ``generic'' toroidal domain
satisfies both nondegeneracy conditions (cf. Definitions~\ref{D:torus} and~\ref{D:torusII}), but proving generic results for
vectorial problems is often extremely hard. Our objective in this
section is to prove an analog of this result in the class of thin toroidal
domains, where one can analyze the harmonic field (and other Beltrami
fields) in detail.

Given a closed smooth curve~$\ga:\TT\to\RR^3$, let us denote by
$\Om(\ga,\ep)$ the toroidal domain (or tube) of thickness~$\ep$
defined by this curve:
\begin{equation}\label{Tgaep}
\Om(\ga,\ep):=\{x\in\RR^3: \text{dist}(x,\ga(\TT))<\ep\}\,.
\end{equation}
The main result of this section can then be stated as follows. When we
say that the result holds for almost all small enough~$\ep$, it means
that there exists some $\ep_0>0$ and a subset $Z\subset[0,\ep_0]$ of
measure zero such that the result holds for all $\ep\in
(0,\ep_0]\backslash Z$.

\begin{proposition}\label{P:isot}
For a generic curve~$\ga$ (in the sense of a dense subset in
the $C^r$~topology, for any $r\geq3$), the toroidal domain
$\Om(\ga,\ep)$ is analytic and nondegenerate of type I and II for almost all small enough $\ep>0$.
\end{proposition}
\begin{proof}
Theorem~\ref{P:tubes} implies that, for a generic curve $\ga$, the
toroidal domain $\Om_1:=\Om(\ga,\ep)$ admits a Beltrami field $B$
satisfying $\curl B=\la_1B$, for some constant $\la_1=O(\ep^3)$, and $\pd\Om_1$
is an analytic Diophantine invariant torus of $B$. This immediately
establishes the property~(i) in Definitions~\ref{D:torus} and~\ref{D:torusII}.

Let us next prove condition~(ii) in Definition~\ref{D:torus}.
In the following computations we shall use the formulas and results obtained in~\cite{Acta} without further mention. Consider the coordinates $(\alpha,r,\theta):\Om_1\to \TT \times (0,1) \times \TT$ that parametrize $\Om_1$ (we are assuming, without any loss of generality, that the length of the core curve $\ga$ is $2\pi$); in particular, $\pd\Om_1$ corresponds to the surface $\{r=1\}$. In these coordinates, $Y:=B|_{r=1}$ takes the form
\[
Y=\pd_\al-\tau(\al)\pd_\theta+\cO(\ep)\,,
\]
where $\tau$ is the torsion of $\ga$, and $\cO(\ep)$ stands for a quantity whose $C^m$ norm is bounded by $C|\ep|$, for any fixed integer $m$. Using the expression of the metric on $\pd\Om_1$ induced from $\RR^3$, in coordinates $(\al,\theta)$, we can readily compute $|Y|^2$, which leads to
\begin{equation}\label{eq.Y}
|Y|^2=1+\cO(\ep)\,.
\end{equation}
Moreover, we can compute the linearizing coordinates $(\vp_1,\vp_2)$ in terms of $(\al,\theta)$ as
\[
\vp_1=\al+\cO(\ep)\,, \qquad \vp_2=\theta+\int_{0}^{\al}\tau(s)ds-[\tau]_{\TT^1}\al +\cO(\ep)\,,
\]
and hence $Y$ takes the following form in these coordinates:
\[
Y=\pd_{\vp_1}-[\tau]_{\TT^1}\pd_{\vp_2}+\cO(\ep)\,.
\]
This implies that the frequency vector of $Y$ (which is Diophantine) is given by
\[
\om_1=1+O(\ep)\,, \qquad \om_2=-[\tau]_{\TT^1}+O(\ep)\,.
\]
Accordingly, we infer from Equation~\eqref{eq.Y} that the function $\cR$ in Definition~\ref{D:torus} solves the cohomological equation
\[
(1+O(\ep))\pd_{\vp_1}\cR+(-[\tau]_{\TT^1}+O(\ep))\pd_{\vp_2}\cR=(1+O(\ep))-|Y|^2=\cO(\ep)\,.
\]
It then follows that $\cR=\cO(\ep)$.

Putting together all these computations, we obtain a matrix $M$, cf.~Equation~\eqref{eq.gen}, of the form

\begin{equation*}
M=\int_{\TT^2}G^{-1}_\ep\cdot\left(
    \begin{array}{cc}
      1+\cO(\ep) & \cO(\ep) \\
      \cO(\ep) & 1+\cO(\ep) \\
    \end{array}
  \right)d\vp\,,
\end{equation*}
where $G_\ep$ is the matrix of the metric (which depends on $\ep$) in
$\vp$-coordinates. It is obvious that $M$ is invertible because the
(inverse) metric matrix $G^{-1}_\ep$ is positive definite, which
shows that $\Om_1$ is nondegenerate of type I. In
fact, \cite[Theorem 7.8]{Acta} ensures that the
twist constant $T$ of the invariant torus $\pd \Om_1$ of $B_1$ (see Definition~\ref{def:ndeg}) is
\begin{equation}\label{twistfinal}
T=c_\ga \ep^2+O(\ep^3)\,,
\end{equation}
where $c_\ga$ is certain explicit constant that depends on the
curve~$\ga$ through its curvature and torsion but not on~$\ep$ nor $\la$. This
constant is nonzero for a generic curve~$\ga$~\cite[Lemma 7.9]{Acta}.

To conclude, let us establish property~(ii) in
Definition~\ref{D:torusII}. Straightforward computations using the
formulas for thin tubes derived in~\cite{Acta} show that
\[
  |\al\cdot
  \om^\perp |\leq C\,,\qquad |n|\geq C\ep\,,
\]
for some $\ep$-independent constant $C$. Here $\al$ is the $\RR^2$-valued function introduced in Definition~\ref{D:torusII}, and not an angular coordinate. One can then use~\eqref{twistfinal} to see that
\[
T+\la\,\bigg[\frac{\al\cdot \om^\perp}{|n|^2}\bigg]_{\TT^2}= c_\ga
\ep^2+ O\bigg(\frac{\la}{\ep^2}\bigg)+O(\ep^3)\,.
\]
Therefore,
\[
T+\la\,\bigg[\frac{\al\cdot \om^\perp}{|n|^2}\bigg]_{\TT^2}\neq0
\]
provided that the nonzero constant~$\la$ satisfies $|\la|<
C|c_\ga|\ep^4$ for a certain positive constant~$C$ independent of~$\ep$.
\end{proof}
\begin{remark}
The nondegenerate Beltrami field $B$ in the proof of Proposition~\ref{P:isot} has eigenvalue $\la_1=O(\ep^3)$. In particular, the ratio between the $L^2$ norms of the plasma current $\|\curl B\|_{L^2}$ and the magnetic field $\|B\|_{L^2}$, in the toroidal domain $\Om(\ga,\ep)$, is small. Since the Beltrami equation is linear, we can safely normalize the magnetic field so that $\|B\|_{L^2}=1$, thus implying that the plasma current is small, which is a desirable feature for stellarator design~\cite{Bo}.
\end{remark}

As any toroidal domain is isotopic to a thin tube, we infer that there
are MHD equilibria of the kind described in Theorems~\ref{T:main}
and~\ref{T:main2} that have arbitrary topology:

\begin{corollary}\label{C.main}
There exist piecewise smooth MHD equilibria with fixed or free
toroidal boundaries of arbitrary topology. More precisely, given any
toroidal domain $\Om_0\subset\RR^3$, one can assume that the domains
$\Om_k$ of Theorem~\ref{T:main} and $\Om,\Om'$ of
Theorem~\ref{T:main2} are diffeomorphic to~$\Om_0$.
\end{corollary}
\begin{proof}
It is an immediate consequence of Theorems~\ref{T:main}
and~\ref{T:main2} and of the fact that, by Proposition~\ref{P:isot}, a generic thin tube is nondegenerate
of type I and II.
\end{proof}

\section{Lipschitz continuous force-free fields with nonconstant factor}\label{S:ff}

A {\em force-free field}\/ in a domain $\Om$ is a vector field $B$ that satisfies the equations
\[
\curl B=f\,B\quad\text{in }\Om\,, \qquad \Div B=0\quad\text{in }\Om\,,
\qquad  B\cdot N=0\quad\text{on }\pd\Om
\]
for some scalar function~$f$. In the context of hydrodynamics, these fields are called Beltrami fields with nonconstant factor~\cite{MYZ,ARMA}.

It is obvious that a force-free field satisfies the MHD equations in
$\Om$ with $P=0$. This is used to define force-free fields of low
regularity. Specifically, one says
that a vector field
$B\in L^2(\Om)$ is {\em force-free}\/ if
\[
\int_{\Om} \left[(B\otimes B)\cdot \nabla w- \frac12|B|^2\Div w\right]\, dx=0\quad
\text{and}\quad \int_{\Om}B\cdot \nabla\phi\,dx=0
\]
for any vector field $w\in C^1_c(\Om)$ and any scalar function
$\phi\in C^1(\Om)$.

The strategy of the proof of Theorem~\ref{T:main} can be readily
adapted to show the existence of Lipschitz-continuous
force-free fields with nonconstant factor on toroidal domains of any
topology. More precisely, one can prove the following:

\begin{theorem}\label{T:ff}
  Let $B_1$ be a nondegenerate Beltrami field of type I with eigenvalue~$\la_1$
  on an analytic toroidal domain~$\Om_1$. For any $N\geq2$ and almost
  all distinct constants $\{\la_k\}_{k=2}^N$, there
  exists a family of nested analytic toroidal domains
  $\{\Om_k\}_{k=1}^N$ as in Theorem~\ref{T:main} and a Lipschitz
  continuous vector field~$B$ satisfying the equation
\[
\curl B=f\,B\,,\qquad \Div B=0
\]
in~$\Om$, where the factor
\[
f:= \sum_{k=1}^N \la_k\,{1}_{\Om_k\backslash
  \overline{\Om_{k-1}}}
\]
is not constant in~$\Om_N$.
\end{theorem}
\begin{proof}
We use the same construction as in the proof of Theorem~\ref{T:main},
but we consider the particular case where $c_k:=0$ for all
$k=2,\cdots,N$. This obviously implies that $b_k=0$ as well. The
effect of this choice of constants is that the vector field $B_k$ in a
neighborhood of $\pd\Om_{k-1}$ is constructed using the
Cauchy--Kovalevskaya theorem with Cauchy datum given by
$B_{k-1}|_{\pd\Om_{k-1}}$, so $\pd\Om_{k-1}$ is no longer a
discontinuity surface of the magnetic field, and that we do not need to use Theorem~\ref{L:cohom}.

For almost all choices of the constant $\la_k$, the Diophantine
invariant torus $\pd\Om_{k-1}$ one obtains is twist. This ensures the
existence of a family of nested toroidal domains $\{\Om_k\}_{k=1}^N$
as above and of a weak solution $(B,P)$ to the MHD equations with
constant pressure $P=0$. In each set
$\Om_k\backslash\overline{\Om_{k-1}}$, $B$ satisfies the equation
$\curl B=\la_kB$ and is analytic up to the boundary; in particular,
$B$ is Lipschitz continuous on~$\Om_N$. The plasma current $\curl B$
is given by Remark~\ref{R:current}, and the singular terms supported
on the toroidal surfaces vanish because~$B$ is continuous. This
formula shows that $B$ is in fact a force-free field with piecewise
constant factor~$f$
as above. The theorem is then proven.
\end{proof}
 \begin{remark}
 The assumption that the matrix $M$ is invertible that appears in the
 definition of nondegenerate toroidal domains of type I is not used to prove
 Theorem~\ref{T:ff} because here we do not need
 Theorem~\ref{L:cohom}. Thus, the theorem holds under the slightly weaker
 assumption on the domain~$\Om_1$ that it admits a Beltrami field for which the boundary is a Diophantine invariant torus.
 \end{remark}

\begin{remark}
In view of Proposition~\ref{P:isot}, Theorem~\ref{T:ff} implies the existence of Lipschitz-continuous
force-free fields with nonconstant factor on thin tubes of any topology.
\end{remark}

\section*{Acknowledgements}

The authors are grateful to Theo Drivas for pointing out the formulation and relevance of the free boundary problem in the context of MHD equilibria. We also thank the reviewers for their careful reading of the manuscript and useful suggestions.
This work has received funding from the European Research Council (ERC) under the European Union's Horizon 2020 research and innovation programme through the grant agreement~862342 (A.E.). It is partially supported by the Swedish Research Council VR grant~2019-04591 (A.L.) and the grants CEX2019-000904-S, RED2018-102650-T and PID2019-106715GB GB-C21 (D.P.-S.) funded by MCIN/AEI/10.13039/501100011033.

\appendix

\section{Analyticity up to the boundary of Beltrami fields}
\label{A.analytic}

In this appendix we show that, if $\Om$ is a bounded,
analytic domain of~$\RR^3$ and $B$ is a Beltrami field on~$\Om$ that
is tangent to the boundary of the domain, then $B$ is analytic up to
the boundary. More generally, we shall prove the following:

\begin{theorem}\label{T.analytic}
  Let $B\in H^1(\Om)$ be a divergence-free field satisfying
  \[
\curl B-\la B= F\;\text{in }\Om\,,\qquad B\cdot N=0\;\text{on }\pd\Om
\]
for some constant~$\la$. If $F\in C^\om(\BOm)$, then $B$ is analytic up to the boundary of the domain.
\end{theorem}

The motivation for this result is that solutions to elliptic boundary
value problems on real analytic domains and with analytic coefficients
are analytic up to the boundary. While curl is not
an elliptic operator, intuitively speaking it is not far from being one provided
that we restrict our attention to divergence-free fields.

Perhaps the most efficient way to make this intuition precise, thereby
proving
Theorem~\ref{T.analytic}, is via a
Dirac-type operator acting on differential forms of mixed degrees. Specifically, let
\[
  \Lab:=\La^0\oplus \La^1 \oplus \La^2\oplus \La^3
\]
be the exterior algebra of~$\RR^3$, where $\La^p$ is the
$\binom3p$-dimensional vector
space of $p$-forms. $\Lab$ is then a real vector space of dimension~8.

Given a differential form $\psi \in H^1(\Om,\Lab)$, we now define a
Dirac-type operator as
\[
\cD\psi:= d\psi+ d^*\psi\,,
\]
where $d$ and $d^*$ respectively denote the exterior derivative and
the codifferential. $\cD$ is an elliptic first order differential operator
satisfying $\cD^2=-\De$ (minus the Hodge Laplacian). The principal symbol of~$\cD$ is the map
$\si_1(\cD): \RR^3\to \cL(\Lab)$ given by
$\si_1(\cD)(\xi)=i\ga(\xi)$, where
\[
\ga(\xi)\Psi:=\xi\wedge \Psi +\star(\xi\wedge\star \Psi)\,,
\]
the wedge denotes the exterior product and $\star$ is the Hodge star operator. Note that $\ga(\xi)^2 =
-|\xi|^2I$.

\begin{proof}[Proof of Theorem~\ref{T.analytic}]
  For each point $x\in\pd\Om$ and for $p=1,2$, let
\begin{align*}
  \Latp&:=\{\Psi \in \La^p: i_{N(x)}\Psi=0\}\,,\\
         \La^1_{x,\mathrm{nor}}&:=\{\Psi \in \La^1: i_\xi\Psi=0\text{ for all }
                \xi\in T_x\pd\Om\}\,,\\
                \La^2_{x,\mathrm{nor}}&:=\{\Psi \in \La^2: i_\xi\Psi=0\text{ for some }
                \xi\in T_x\pd\Om\}\,,
\end{align*}
denote the vector spaces (of dimension $3-p$ and $p$, respectively) of $p$-forms that are tangent and normal to the boundary
at the point~$x$. Here and in what follows,
\[
T_x\pd\Om:=\{\xi\in \RR^3 \text{ with } \xi\cdot N(x)=0\}
\]
is the tangent space of the boundary at the point~$x$ and
$i_\xi\Psi$ denotes the interior product of the
vector~$\xi$ with the form~$\Psi$. For each $x\in \pd\Om$, we now define the
4-dimensional vector subspaces
\[
E_x':=\LatI\oplus \LanII\,,\qquad E_x'':= \La^0\oplus \LanI\oplus \LatII\oplus\La^3\,.
\]
Obviously $\Lab=E_x'\oplus E_x''$ for any $x\in\pd\Om$ and these
subspaces are orthogonal. Let us denote the orthogonal projector
$\Lab\to E_x''$ by $P_x''$ and introduce the trace-type operator
\[
(P''\psi)(x):=P_x''\psi(x)\in H^{1/2}(\pd\Om,\Lab)\,.
\]
Therefore, $P''$ maps each differential form $\psi\in H^1(\Om,\Lab)$ to the
projection of its boundary trace to the space of $H^{1/2}$ sections of
the vector bundle $E''$ over the boundary whose section at~$x$ is
$E_x''$. Equivalently, one can understand $P''$ as the action of a
zeroth order pseudodifferential operator (whose principal symbol is an orthogonal
projector) on the usual boundary trace $\psi|_{\pd\Om}\in H^{1/2}(\pd\Om,\Lab)$.

We are interested in the boundary value problem for the Dirac operator
defined as
\begin{subequations}\label{BVPDirac}
\begin{align}
  \cD\psi -\la \star\psi = G\quad &\text{ in }\Om\,,\label{eqD}\\
  \qquad P''\psi=0 \quad &\text{ on }\Om\,.\label{BCD}
\end{align}
\end{subequations}
The key observation now is that, for each $x\in \pd\Om$ and each
$\xi\in T^*_x\pd\Om$, the operator
\[
a(x,\xi):= \ga(N^\flat(x))\ga(\xi)\in \cL(\Lab)
\]
maps $E_x'$ onto $E_x''$ and viceversa. It is then
well known~\cite[Corollary 3.18]{Bar} that~\eqref{BVPDirac}
defines an elliptic boundary value problem. For the benefit of the
reader, let us mention that this is a straightforward consequence of
the fact that the principal symbol of the Calder\'on projector
associated to the boundary problem~\eqref{BVPDirac} is $\left[
  |\xi|I-ia(x,\xi)\right]/(2|\xi|)$. Since~\eqref{BVPDirac} is an
elliptic boundary value problem, the theory of Boutet de
Monvel~\cite{BdM} ensures that $\psi$ is analytic up to the boundary
if $G\in C^\om(\BOm,\Lab)$.

This implies that the vector field~$B$ is analytic up to the
boundary. In order to see this, let $\psi$ be the differential 1-form
dual to~$B$, which we can regard as an element of $H^1(\Om,\Lab)$. If
$G\in C^\om(\BOm,\Lab)$ denotes the 1-form dual to~$F$, the equations
$\Div B=0$ and $\curl B-\la B= F$ read in terms of the associated
forms as
\[
d^*\psi=0\,,\qquad \star d\psi-\la \psi=G\,.
\]
As $\star\,\star $ is the identity, we infer that~$\psi$ satisfies
Equation~\eqref{eqD} in~$\Om$. Furthermore, the condition $B\cdot N=0$
amounts to saying that the differential form $\psi$ (which is of pure
degree~1) satisfies $i_N\psi=0$. This ensures that~$\psi$ satisfies
the boundary condition~\eqref{BCD}. The analytic regularity result for
solutions to the elliptic problem~\eqref{BVPDirac} then implies that
$B\in C^\om(\BOm)$, as claimed.
\end{proof}

\section{Beltrami fields, Cauchy--Kovalevskaya and thin tubes}\label{Ap1}

In this appendix we review some previous results about Beltrami
fields. The first result is an analog of the Cauchy--Kovalevskaya
theorem for the curl operator which was proved in~\cite{Annals} and
which plays an important role in the proof of our main theorem. The
basic idea is that, even though the curl operator does not admit
noncharacteristic surfaces, a variant of the
Cauchy--Kovalevskaya theorem holds provided that the Cauchy datum
satisfies a certain necessary constraint. In the statement, the reader
should note that the second equation, $\Div B=0$, is a consequence of
the first when $\la\neq0$, so in that case it can be omitted.

\begin{theorem}\label{T:CK}
Let $\Sigma$ be an embedded and oriented analytic surface of
$\RR^3$. For any real constant~$\la$, the Cauchy problem
\begin{equation}
\curl B=\la B\,, \qquad \Div B=0\,,\qquad B|_\Sigma=X\,,
\end{equation}
with $X$ an analytic field tangent to $\Sigma$, has a unique solution in a neighborhood of~$\Sigma$ if and only if
\begin{equation}\label{Eq.CKclosed}
d (j_\Sigma^*X^\flat)=0\,,
\end{equation}
where $X^\flat$ is the metric dual $1$-form of $X$, and $j_\Sigma:\Sigma\to \RR^3$ is the embedding of $\Sigma$ in $\RR^3$.
\end{theorem}

The second result concerns Beltrami fields on thin toroidal domains
$\Om(\ga,\ep)$, as defined in~\eqref{Tgaep}.
The following theorem, which was proved in~\cite{Acta,ELP}, shows that
in a generic thin tube $\Om(\ga,\ep)$ there exist Beltrami fields that
are tangent to the boundary and exhibit some nice ``almost-integrability'' properties.

\begin{theorem}\label{P:tubes}
Let $\ga$ be a generic closed curve, where by ``generic'' we mean an open and dense set, in the $C^3$ topology, of smooth closed curves in
$\RR^3$. For almost every small enough $\ep>0$ and almost every nonzero constant $\la_1=O(\ep^3)$, the boundary of $\Om(\ga,\ep)$
is a Diophantine invariant torus of a Beltrami field $B$, which satisfies
the equation $\curl B=\la_1 B$ in $\Om(\ga,\ep)$. Moreover, the boundary $\pd\Om(\ep,\ga)$ can be
chosen to be analytic for a dense set of curves.
\end{theorem}
\begin{remark}
In fact, the toroidal domain $\Om(\ga,\ep)$ is covered by a set of Diophantine invariant tori~$\mathcal I$ of almost
full measure in the sense that, for any $\de>0$,
$$\frac{|\mathcal I|}{|\Om(\ga,\ep)|}>1-\de\,,$$
provided that $\ep$ is small enough.
\end{remark}

%\section{A KAM scheme to solve cohomological type equations}\label{Ap2}

\bibliographystyle{amsplain}

\end{document}